\pdfoutput=1
\documentclass[mnsc,nonblindrev]{informs3} 

\DoubleSpacedXI 

\usepackage{natbib}
 \bibpunct[, ]{(}{)}{,}{a}{}{,}%
 \def\bibfont{\small}%
\usepackage{color}
\usepackage{algorithm}
\usepackage{algorithmic}
\usepackage{setspace}

\usepackage{array}
\usepackage{multirow}
\usepackage{comment}
\usepackage[toc,page]{appendix}
\usepackage{appendix}
\usepackage{setspace}
\usepackage{hyperref}

\TheoremsNumberedThrough     
\ECRepeatTheorems

\EquationsNumberedThrough    

\newcommand{\bE}{\mathbf{E}}
\newcommand{\R}{\mathbb{R}}
\newcommand{\DLP}{\text{\normalfont{DLP}}}
\newcommand{\LP}{\mathsf{LP}}
\newcommand{\ALG}{\text{\normalfont{ALG}}}
\newcommand{\HO}{\text{\normalfont{HO}}}

\newcommand{\AuxHO}{\mathsf{H}}

\newcommand{\I}{\text{\normalfont{I}}}
\newcommand{\II}{\text{\normalfont{II}}}

\newcommand{\FF}{\mathsf{ALG2}}
\newcommand{\FD}{\mathsf{ALG3}}
\newcommand{\FR}{\mathsf{ALG4}}
\newcommand{\MR}{\mathsf{ALG6}}
\newcommand{\LPT}{\mathsf{ALG5}}

\newcommand{\real}{\mathbb{R}}

\newcommand{\calA}{\mathcal{A}}
\newcommand{\calR}{\mathcal{R}}

\usepackage{xcolor}

\begin{document}

\RUNAUTHOR{Sun, Wang, and Zhou}
\RUNTITLE{Near-Optimal Primal-Dual Algorithms for NRM}

\TITLE{Near-Optimal Primal-Dual Algorithms for Quantity-Based Network Revenue Management}

\ABSTRACT{
We study the canonical quantity-based network revenue management (NRM) problem where the decision-maker must irrevocably accept or reject each arriving customer request with the goal of maximizing the total revenue given limited resources. The exact solution to the problem by dynamic programming is computationally intractable due to the well-known curse of dimensionality. Existing works in the literature make use of the solution to the deterministic linear program (DLP) to design asymptotically optimal algorithms. Those algorithms rely on repeatedly solving DLPs to achieve near-optimal regret bounds. It is, however, time-consuming to repeatedly compute the DLP solutions in real time, especially in large-scale problems.

In this paper, we propose innovative algorithms for the NRM problem that are easy to implement and do not require solving any DLPs. Our algorithm achieves a regret bound of $o(\sqrt{T})$, where $T$ is the system size. To the best of our knowledge, this is the first NRM algorithm that (i) has an $o(\sqrt{T})$ asymptotic regret bound, and (ii) does not require solving any DLPs.

}

\ARTICLEAUTHORS{
\AUTHOR{Rui Sun}\AFF{Alibaba Group US,
Bellevue, WA, 98004,
\EMAIL{sunruimit@gmail.com}}
\AUTHOR{Xinshang Wang}\AFF{Alibaba Group US,
Bellevue, WA, 98004,
\EMAIL{xinshang.w@alibaba-inc.com}}
\AUTHOR{Zijie Zhou}\AFF{Operations Research Center,
Massachusetts Institute of Technology, Cambridge, MA, 02139,
\EMAIL{zhou98@mit.edu}}
}

\maketitle

\section{Introduction}

In this paper, we consider the canonical network revenue management (NRM) problem
\citep{williamson1992airline, talluri1998analysis}, 
which finds applications in airline, retail, advertising, and hospitality \citep{ talluri2004revenue}. 
The NRM problem is stated as follows: there is a set of resources with limited capacities that are available over a finite time horizon. Customers of different types arrive sequentially over time. The type of a customer is defined by the customer's consumption of resources and the price she pays. Each type of customer requests a certain amount of each resource. 
Upon a customer's arrival, the decision-maker needs to decide whether to accept or reject the customer. 
Accepting the customer generates a certain amount of revenue, which equals the price the customer pays, and also consumes certain units of each resource associated with the customer's type. Rejecting the customer generates no revenue and uses no resources. The objective is to maximize the total expected revenue during the entire horizon under the capacity constraints of all the resources. 

The model we focus on in this paper is known as the ``quantity-based'' NRM model. There exists another well-known NRM model called the ``price-based'' model. In the price-based NRM problem, the decision-maker chooses product prices rather than admission (accept/reject) decisions. These two models are different but are equivalent in certain special cases, such as the single resource multi-product model suggested in \cite{maglaras2006dynamic}.


A classic example of the NRM problem is airline seat allocation. In this problem, resources correspond to flight legs and the capacity of a resource corresponds to the number of seats on the flight.
Each customer requests a flight itinerary that consists of a single or multiple flight legs (thus a network structure) and pays the ticket price of the itinerary. Arriving customers are sorted into different classes based on combinations of itinerary and price. This allows the airline to maximize the expected revenue by allocating seats on different flights to different customer classes. 

In theory, the NRM problem can be solved by dynamic programming. The state is described jointly by time and the remaining capacities of all the resources. 
The optimal admission policy compares the revenue of each arriving customer with the marginal value of the requested resources and accepts the customer if the former is larger. However, the time complexity of this optimal algorithm grows exponentially with the number of resources, and is thus computationally intractable. In this paper, we study approximate algorithms that are easy to implement and have provable performance guarantees.


One of the most used mathematical tools for designing and analyzing approximate algorithms in the NRM literature is the deterministic linear program (DLP) formulation. See Section \ref{sec:review} for a detailed review of algorithms based on solutions of the DLP. Typically, the DLP is formulated by replacing random variables in the model with their distributional information (e.g., replacing random demand variables with their expectations). 

The DLP formulation is essential in the following two aspects.
First, an optimal solution to the DLP is a guideline for making online decisions. Specifically, the primal solution of the DLP can be used to design booking limit heuristics, nesting strategies
\citep{talluri2004revenue}, and probabilistic assignment policies \citep{reiman2008asymptotically, jasin2012re, bumpensanti2020re}.
In addition, the dual solution of the DLP serve as bid prices in bid-price control heuristics \citep{talluri1998analysis}. 
Second, the DLP provides an upper bound on the optimal revenue since the capacity constraints in the DLP are only satisfied in expectation. The DLP upper bound (and its variants) is considered the benchmark performance in the analysis of many algorithms.

    
    
    



In the literature of quantity-based NRM, \citet{reiman2008asymptotically}, \citet{jasin2012re}, and \citet{bumpensanti2020re} show much better regret bounds compared to preceding results by (re-)solving DLPs multiple times. 
More precisely,
 \cite{reiman2008asymptotically} show that strategically resolving the DLP once achieves a regret bound of $o(\sqrt{T})$, where $T=1,2\ldots$ is the system size, improving upon the classic $O(\sqrt{T})$ bounds \citep{talluri2004revenue}. \cite{jasin2012re} show that  resolving the DLP in every time period achieves a regret bound of $O(1)$.
\cite{bumpensanti2020re} prove a uniform regret bound of $O(1)$ that does not depend on any numerical values in the solutions of the DLP. 


Despite the theoretical developments for the NRM problem, there are challenges in implementing those algorithms that require solving DLPs in an online manner.
 First, direct linear system solvers, which are necessary for solving DLPs, are time consuming for large-scale problems. Second, consider  e-commerce platforms that must deal with more than hundreds of customer visits per second \footnote{For example,  62 million monthly visits took place on Expedia in 2018 (\cite{expedia}), and 193 million bookings on Airbnb in 2020 (\cite{airbnb})}. In a typical implementation of online algorithms on such e-commerce platforms, DLPs are solved on dedicated computational servers, which sync with an online streaming database at a much lower rate. In such cases, algorithms handling the online customer traffic are not able to acquire DLP solutions to real-time models.



To overcome these difficulties, we propose primal-dual algorithms that are amenable to online platforms. Our algorithm does not require solving any linear programs or performing any matrix factorization, and still achieves an $o(\sqrt{T})$ regret bound.

Similar motivations for such fast online algorithms that do not require solving any linear programs are also considered in \cite{li2020simple}.
Specifically, \cite{li2020simple} study the online linear programming model, which is a generalization of the quantity-based NRM model studied in this paper. We show that our algorithm achieves a regret bound of $o(\sqrt{T})$, while the algorithm in \cite{li2020simple} for the generalized problem achieves a regret bound of $O(\sqrt T)$. We will discuss more differences between our work and \cite{li2020simple} in Section~\ref{sec:review}.

\subsection{Overview of main results}

Table \ref{table1} summarizes the related regret bounds proved in the literature, and our results in this paper. In the table, $T$ stands for the scaling factor, namely the size of the system.
 
The algorithms proposed in \cite{reiman2008asymptotically,jasin2012re, bumpensanti2020re} are re-solving algorithms for the NRM problem. The regret bound proved in \citet{reiman2008asymptotically} grows slower than $O(\sqrt{T})$, and the regret bounds proved in \citet{jasin2012re, bumpensanti2020re} are $O(1)$. 
The main intuition behind their algorithms is as follows: if the algorithm makes a mistake of accepting too many customers of a certain type in an earlier stage, then after re-solving the DLP, the algorithm is able to correct the mistake by accepting fewer customers of that type in later periods. \cite{li2019online} also prove an $o(\sqrt{T})$ regret bound but for a more general online linear programming model. All those algorithms that obtain $o(\sqrt{T})$ regret bounds require solving the DLP at least twice. In particular, at least one of the DLPs has to be solved online. That is, the algorithms need to wait for the DLP solution before making any further online decisions.

In this paper, we propose three two algorithms that do not require solving any DLPs for the NRM problem. Both algorithms are based on primal-dual frameworks, and they apply online convex optimization \citep{hazan2019introduction} techniques to learn the optimal dual values (or bid prices) of the resources.
We introduce our first algorithm (Algorithm \ref{alg:ff}) in Section~\ref{sec:ff}.  
We show that Algorithm~\ref{alg:ff} has a regret bound of $O(\sqrt{T})$. This algorithm is a warm up for our subsequent main results. It can also be viewed as a direct adaptation of the existing algorithms, such as \cite{li2019online}, \cite{li2020simple}, \cite{balseiro2020best},
\cite{balseiro2020dual},
\cite{agrawal2014fast}
and \cite{agrawal2016linear}.

In Section~\ref{sec:fd}, 
we introduce a new algorithm (Algorithm~\ref{alg:fd}). It builds on Algorithm~\ref{alg:ff} and uses a ``thresholding'' technique for making admission decisions. The algorithms in \cite{bumpensanti2020re} use a similar thresholding technique. Their algorithm relies on repeatedly solving DLPs to calculate the threshold values.
By contrast, in our algorithm, the threshold values are adaptively computed based on a novel statistical property of the admission decisions. 
We show that Algorithm~\ref{alg:fd} has a regret bound of $O(T^{3/8} (\log T)^{\frac{5}{4}})$. 
To the best of our knowledge, Algorithm~\ref{alg:fd} is the first algorithm, for any dynamic multi-resource allocation problem, that (i) has an $o(\sqrt{T})$ asymptotic regret bound and (ii) does not require solving any DLPs (or equivalently time-consuming problems).



In Section~\ref{sec:numerical},
we provide results from numerical experiments to demonstrate the performance of all these algorithms.

\begin{table}[ht] 
\small
\caption{List of regret bounds for models applicable to NRM}\label{table1}
\begin{center}
\begin{tabular}{ | m{5cm} | m{3cm} | m{3cm}| m{4cm} | }
\hline
Previous work       & Arrival assumptions & Regret bound & Number of times solving LPs \\
\hline
\citet{reiman2008asymptotically} & known i.i.d. &  $o(\sqrt T)$                      & $2$                                    \\
\hline
\citet{jasin2012re}    & known i.i.d.    &     $O(1)$                  & $ \Theta(T)  $                              \\
\hline
\citet{bumpensanti2020re}   & known i.i.d.   &    $O(1)$                   &   $\Theta(\log \log T)$                                    \\
\hline
\cite{agrawal2014fast}  & unknown i.i.d.        &     $O(\sqrt{T})$                  & 1  \\
\hline
\citet{li2019online}     & unknown i.i.d.          &     $O(\sqrt T)$                  &  $\Theta(\log T)$              \\
\hline
\citet{li2019online}     & unknown i.i.d.          &     $O(\log T \log \log T)$                  & $\Theta(T)$  \\
\hline
\citet{li2020simple}      & unknown i.i.d.   \& random permutation      &     $O(\sqrt T)$                  & $0$   \\
\hline
\citet{balseiro2020best, balseiro2020dual} & unknown i.i.d.  &     $O(\sqrt T)$                  & $0$   \\

\hline

\end{tabular}
\end{center}
\end{table}




\begin{table}[ht]
\small
\begin{center}
\begin{tabular}{ | m{6cm} | m{3cm}| m{5cm} | }
\hline
Our algorithms       & Regret bound & Number of times solving LPs \\
\hline
Algorithm~\ref{alg:ff}  &  $O(\sqrt T)$  &  $0$ \\
\hline
Algorithm~\ref{alg:fd} &  $o(\sqrt T)$  &  $0$ \\
\hline
\end{tabular}
\end{center}
\end{table}

\subsection{Other related work}\label{sec:review}

We study in this paper a canonical network revenue management problem that has a quantity-based formulation. Compared with a 
price-based formulation as studied in \cite{gallego1994optimal, gallego1997multiproduct},
both formulations can be solved using dynamic programming (DP). However, the curse of dimensionality renders the DP computationally intractable even for moderate size problems. Therefore, this has motivated studies in the literature to find approximate algorithms for both formulations.

In the price-based formulation, the decision-maker chooses posted prices of different products. In this stream of literature, \cite{gallego1997multiproduct} show that a static fixed-price policy is asymptotically optimal on the fluid scale (see the definitions of fluid scale optimality and diffusion scale optimality in Section~\ref{sec:asymptotic}). It achieves a regret bound of $O(\sqrt{T})$ given system size $T$.
\cite{maglaras2006dynamic} and \cite{chen2013simple} study the potential benefit of re-optimization to improve the static control policy. \cite{jasin2012re} and \cite{atar2013asymptotically} propose policies that are optimal on the diffusion scale. 



For the quantity-based formulation, a review of DLP-based algorithms can be found in the book by \cite{talluri2004revenue}.
Here, we briefly review a few commonly used control policies: booking limit control, nesting and bid-price control.

\textbf{Booking limit control} sets a fixed quota for each customer type and accepts customers in a first-come-first-serve (FCFS) fashion. The booking limits are given by DLP solutions, as proposed in \cite{williamson1992airline}. Later,
many variants of the DLP are proposed, e.g., \cite{wollmer1992airline} and \cite{li2004control}, to improve the practical performance of the booking limit control policy.

\textbf{Nesting} is a remedy strategy for booking limit control. The nesting policy ranks different customer types based on their revenue and resource usage, and allows high-ranking types to use the quota of the low-ranking types. The nesting policy is shown to be effective for single-resource problems. However, when there exist multiple resources, \cite{talluri1998analysis} show that the advantage of nesting is less clear due to the ambiguity in ranking different customer types.

\textbf{Bid-price control} uses dual variables of DLPs to decide admissions. Bid prices are defined as the Lagrangian multipliers associated with the capacity constraints. A customer is accepted if the price she pays is higher than the estimated value of the requested resource. \cite{talluri1998analysis} provide a comprehensive analysis on the asymptotic optimality of bid-price control, and a comparison of different methods to estimate bid prices.

The quantity-based NRM model in this paper is also related to the online knapsack/secretary problems in \cite{arlotto2018logarithmic} and \cite{arlotto2019uniformly}, the packing/matching problems in \cite{vera2020bayesian}, and the online linear programming (OLP) problems in \cite{agrawal2014dynamic},  \cite{li2019online}, and \cite{li2020simple}. 

\cite{arlotto2019uniformly} study a multi-secretary problem, where the decision-maker selects from a sequence of i.i.d. random variables with known distribution to maximize the expected value of the sum of the selected variables under a given budget constraint. 
\cite{arlotto2019uniformly} develop an adaptive algorithm that makes accept/reject decisions by comparing the ratio between the residual budget and the remaining number of arrivals to certain thresholds, and prove that the algorithm achieves a uniformly bounded regret compared to the optimal offline policy. 

\cite{vera2020bayesian} study an online allocation problem that generalizes a wide range of online problems including multi-secretary, online packing/matching and NRM. They propose an algorithm that resolves the DLP upon each customer arrival and  accepts a customer if the acceptance probability suggested by the DLP is greater than $0.5$. \cite{vera2020bayesian} show that their algorithm achieves $O(1)$ regret under mild assumptions on the customer arrival process. The analysis is based on an innovative ``compensated coupling" technique. The design of their algorithm and their proof idea are generally different from what we present in this paper.

The OLP problem studied in
\citet{li2019online} and \citet{li2020simple}
generalizes the NRM problem as the OLP model does not impose any parametric structure on the distribution of customer types. In addition, the OLP model assumes the revenue and resource units associated with each arriving customer are i.i.d. sampled from an unknown distribution. \citet{li2019online} propose an algorithm that solves approximate dual problems of the DLP at geometric time intervals, and show the algorithm achieves a regret bound of $O(\sqrt{T})$. 
They also propose a resolving heuristic that updates the solution of the dual problem for each customer arrival, and show the resolving algorithm achieves a regret bound of $O(\log{T} \log{\log {T}})$.
Later, \citet{li2020simple} further the study aiming to relax the assumptions about the input data proposed in \citet{li2019online}. \citet{li2020simple} provide an algorithm that does not require solving any LPs, and show that their algorithm achieves a regret bound of $O(\sqrt{T})$ without assumptions that ensure a strong convexity of the dual problem. 

It is also worth mentioning that \cite{agrawal2014fast} study fast algorithms for a class of online convex programming problems. The algorithms in \cite{agrawal2014fast} require solving a DLP once in order to estimate an upper bound on the norm of dual variables. Note that in both our model and the model in \citet{li2020simple}, constant upper bounds on the norm of dual variables can be derived due to a convenient assumption that the constraint capacities are non-negative and scale with the size of the time horizon.

\section{Problem Formulation}\label{sec:setting}

Consider a finite time horizon of $T$ periods. There are $n$ types of customers, indexed by $j\in[n]$. (Throughout the paper, we use $[k]$ to denote the set $\{1, 2, \ldots, k\}$ for any positive integer $k$.) In each period, one customer arrives, and the type of the customer is $j$ with probability $\lambda_j$. We must have $\sum_{j \in [n]} \lambda_j = 1$. 
Let $\Lambda_j(t)$ denote the number of arrivals of type-$j$ customers during periods $1,2,\ldots,t$, for $t \in [T]$,
and $\Lambda_j(t_1, t_2)$ the number of arrivals of type-$j$ customers during periods $t_1, \ldots, t_2 $
for $t_1 < t_2$ and $t_1, t_2 \in [T]$.
There are $m$ resources, indexed by $i\in[m]$. Resource $i$ has initial capacity $C_i > 0$. Let $C=[C_1,\ldots,C_m]^\top$ be the vector of initial capacities. 
Upon the arrival of each customer, we need to make an irrevocable decision on whether to accept or reject the customer. Accepting a customer of type $j$ generates revenue $r_j$ and consumes $a_{ij}$ units of each resource $i$. Let $r =[r_1,\ldots,r_n]^\top$ denote the vector of revenues, $A_j = [a_{1j}, \ldots, a_{mj}]^\top$ the column vector of resource consumption associated with customer type $j$, and $A=[A_1, \ldots, A_n] \in \real^{m \times n}$ the bill-of-materials (BoM) matrix. Rejecting a customer generates no reward and consumes no resource. 
Remaining resources at the end of the horizon have no salvage value. The objective is to maximize the total expected revenue during the entire horizon by deciding whether or not to accept each arriving customer while satisfying all the capacity constraints of the resources.

For an online algorithm $\ALG$, let $x_j^{\ALG}(s_1,s_2)$ denote the number of type-$j$ customers accepted during periods $s_1, \ldots, s_2 $ under the algorithm, for all $j\in[n]$, $s_1, s_2 \in [T]$ and $s_1 < s_2$. 
An online algorithm is \emph{feasible} if it is non-anticipating and satisfies
\[
\sum_{j=1}^n A_j x^\ALG_j(1,T) \leq C, \text{ and } x^\ALG(s_1,s_2) \leq \Lambda_j(s_1,s_2), \; \forall j\in[n], \; \forall s_1, s_2 \in [T], s_1 < s_2. 
\]
The total revenue of ALG is given by
\[ V^{\ALG} := \sum_{j=1}^n r_j x_j^{\ALG}(1,T).\]

\textbf{Hindsight optimum.}
To characterize how close ALG is to the ``best'' algorithm, we compare $V^{\ALG}$ with the \emph{hindsight} optimal revenue. The hindsight optimal revenue is defined as the revenue of an optimal algorithm that knows the arrival information of all customer types a priori.
Let $V^{\HO}$ denote the hindsight optimal revenue. Formally, we have
\begin{align}\label{eq:HO}
\begin{split}
V^{\HO} := & \max_z \sum_{j=1}^n r_j z_j \\
\text{s.t. } & \sum_{j=1}^nA_j z_j \leq C\\
& 0 \leq z_j \leq \Lambda_j(1,T) \quad \forall j \in [n].
\end{split}
\end{align}
Let $\bar{z}_j$ denote an optimal solution to \eqref{eq:HO}. Then $V^{\HO}$ is given by
$V^{\HO} = \sum_{j=1}^n r_j \bar{z}_j$.
It is straightforward that $V^\ALG \leq V^\HO$ because $x_j^{\ALG}(1,T)$ for $j\in[n]$ is always a feasible solution to \eqref{eq:HO}. Note that $V^{\HO}$ and $\bar{z}_j$ for $j\in[n]$ are random variables that depend on $\Lambda_j(1,T)$. The \emph{hindsight optimum} is defined as the expectation of the optimal hindsight revenue, i.e., $\bE[V^{\HO}] = \bE[\sum_{j=1}^n r_j \bar{z}_j]$.



In this paper, we define the \emph{regret} of an algorithm ALG as
\begin{equation}\label{eq:regret}
    \bE[V^{\HO} - V^{\ALG}],
\end{equation}
namely the gap between the expected total revenue of ALG and the hindsight optimum. We focus on analyzing \emph{asymptotic upper bounds} on the regret of online algorithms. We will define the asymptotic regime of our model in Section~\ref{sec:asymptotic}.

\textbf{Deterministic linear program (DLP).}
The DLP formulation is a useful mathematical tool for analyzing regret bounds, and it is obtained by replacing all random variables with their expectations:
\begin{align}
\begin{split} \label{eq:DLP}
V^{\DLP} := & \max_w \sum_{j=1}^n r_j w_j \\
\text{s.t. } & \sum_{j=1}^n A_j w_j \leq C\\
& 0 \leq w_j \leq \lambda_j T, \quad \forall j \in [n] .
 \end{split}
 \end{align}
Let $w_j^*$ for $j\in[n]$ be an optimal solution to \eqref{eq:DLP}. We have $V^{\DLP}=\sum_{j=1}^n r_j w_j^*$.

The constraints in \eqref{eq:DLP} are weaker than those of the hindsight problem since the capacity constraints in \eqref{eq:DLP} only need to be satisfied in expectation. In addition, observe that the expectation $\bE[\bar{z_j}]$ of the hindsight optimal solution $\bar z_j$ is a feasible solution to \eqref{eq:DLP}. It is thus easy to verify that
\begin{equation} \label{eq:vho_vdlp}
\bE[V^{\ALG}] \leq \bE[V^{\HO}] \leq V^{\DLP}.
\end{equation}
We refer the reader to \citet{talluri1998analysis} for a detailed proof of this result.

Equivalently, we can reformulate \eqref{eq:DLP} as
\begin{align}
\begin{split} \label{eq:DLP1}
V^{\DLP} = & \max_u\ T \sum_{j=1}^n r_j u_j \\
\text{s.t. } & \sum_{j=1}^n A_j u_j \leq \frac{C}{T}\\
& 0 \leq u_j \leq \lambda_j , \quad \forall j \in [n] ,
 \end{split}
 \end{align}
 where the decision variable $u_j$ stands for the (fractional) number of type-$j$ customers accepted per period.

\subsection{Asymptotic regime}\label{sec:asymptotic}


Define $\bar T$, $\bar C_1, \bar C_2,\ldots, \bar C_m$ as the \emph{unscaled} horizon length and capacity levels, respectively. Let $k$ denote the scaling factor. In the asymptotic regime, we have $T:= k \cdot \bar T$, and $C_i := k \cdot \bar C_i$ for all $i \in [m]$. By the definition in \citet{reiman2008asymptotically}, an algorithm ALG is optimal on the fluid scale if its regret grows slower than $k$, namely
\[ 
\lim_{k \to \infty} \frac{\bE[V^\HO - V^\ALG]}{k} = 0,
\]
and an algorithm ALG is optimal on the diffusion scale if its regret grows slower than $\sqrt{k}$, namely
\[ 
\lim_{k \to \infty} \frac{\bE[V^\HO - V^\ALG]}{\sqrt{k}} = 0.
\]

We are interested in analyzing the regret bounds of algorithms as $k$ increases. In the rest of the paper, we express the regret bounds in terms of $T$, with the understanding that $T$ is proportional to $C_1,C_2,\ldots,C_m$.





\section{Optimal Algorithm on the Fluid Scale}\label{sec:ff}

In this section, we propose an algorithm that is optimal on the fluid scale (Algorithm~\ref{alg:ff}). The algorithm is a bid-price control heuristic that uses \emph{online gradient descent} (OGD) to update the bid prices of the resources. 



\subsection{Preliminaries in online convex optimization}

We first introduce the \emph{online gradient descent} (OGD) algorithm (see \citet{hazan2019introduction} for a review). In an online convex optimization (OCO) problem, we need to make a sequence
of decisions $x_1, x_2,\ldots$ from a fixed feasible set $\mathcal{K}$. After each decision $x_t$ is chosen in period $t$, it encounters a convex cost function $g_t$. The objective is to minimize the sum of the sequence of convex cost functions. 
In each period $t$, OGD takes a step of size $\eta_t$ from the current point $x_t$ in the negative (sub)gradient direction of the cost function $g_t$ to obtain $x_{t+1}'$, and then project the point to the convex set $\mathcal{K}$ to obtain $x_{t+1}$. We outline the detailed procedure of OGD in Algorithm~\ref{alg:ogd}.
The algorithm's performance is measured by a special notion of regret that is defined as
\begin{equation}\label{eq:regretOGD}
\text{Regret} = \sum_{s=1}^t g_s(x_s)-\min_{x \in \mathcal{K}}\sum_{s=1}^t g_s(x) .    
\end{equation}
Note that the definition of regret in \eqref{eq:regretOGD} for the online convex optimization problem is different from the definition of regret in \eqref{eq:regret} for the NRM problem.

Let $\Vert \cdot \Vert$ denote the $L_2$ norm. Let $D$ denote an upper bound of the diameter of the convex set $\mathcal{K}$ such that $\max_{x,x' \in \mathcal{K}} \| x - x'\| \leq D$. Let $G$ be an upper bound of the maximum gradient of the convex functions $g_s$ at each point $x_s$ such that $ \max_{s\in[t]} \Vert \nabla g_s(x_s) \Vert \leq G$. The following proposition is well-established in the literature.

\begin{proposition} \label{prop:ogd}
\citep{hazan2019introduction} The online gradient descent (OGD) algorithm as shown in Algorithm~\ref{alg:ogd} with step sizes $\eta_s = \frac{D}{G\sqrt{s}}$ satisfies 
\begin{equation}
    \sum_{s=1}^t g_s(x_s)-\min_{x \in \mathcal{K}}\sum_{s=1}^t g_s(x) \leq \frac{3}{2}GD\sqrt{t}.
\end{equation}
\end{proposition}

\begin{algorithm}[h]
\caption{\citep{hazan2019introduction} Online Gradient Descent (OGD)}
\begin{algorithmic}[1]\label{alg:ogd}
\footnotesize
\STATE Input: convex set $\mathcal{K}$, time horizon $t$, initial point $x_1 \in \mathcal{K}$, sequence of step sizes $\{\eta_s\}_{s\in[t]}$ 
\FOR{$s = 1,\ldots,t$}
\STATE choose $x_s$; 
\STATE observe convex cost function $g_s:\mathcal{K} \mapsto \R$ and cost $g_s(x_s)$
\STATE update and project:
\begin{align*}
x_{s+1}'&=x_s-\eta_s \nabla g_s(x_s) \\    
x_{s+1}&=\Pi_{\mathcal{K}}(x_{s+1}')
\end{align*}
\ENDFOR
\end{algorithmic}
\end{algorithm}

\subsection{Connection Between OCO and NRM}
In this section, we discuss how to apply OGD to our revenue management model. The connection between the OCO problem and the NRM problem is a specialized bid-price heuristic. Consider the Lagrangian relaxation of DLP \eqref{eq:DLP1}
\begin{align}\label{eq:dualLP}
\begin{split} 
\mathsf{LP}(\theta) = \max_x &\;\; T \sum_{j=1}^n r_j x_j + T \sum_{i=1}^m \theta_i( \frac{C_i}{T} - \sum_{j=1}^n x_j a_{ij}) \\
\text{s.t. } &\;\; 0 \leq x_j \leq \lambda_j, \quad \forall j \in [n] .
\end{split}
\end{align} 

We view $\theta_i$, namely the Lagrangian multiplier for capacity constraint $C_i$, as the bid price for resource $i$. Let $\theta_i^{(t)}$ denote the bid price for resource $i$ at the beginning of period $t$. Let $j(t)$ denote the type of customer arriving in period $t$. The customer in period $t$ is accepted if and only if $r_{j(t)}$ is larger than the aggregated cost $\sum_{i=1}^m \theta_i^{(t)} a_{i,j(t)}$.  

In the corresponding OCO problem, the bid-price vector $(\theta_1^{(t)}, \theta_2^{(t)}, \ldots, \theta_m^{(t)})$ is the decision in period $t$. Once the bid-price vector is chosen for period $t$, the realized cost function is

\begin{equation} \label{eq:gtthetaC}
    g_t(\theta^{(t)}) = \sum_{i=1}^m \theta_i^{(t)} \big( \frac{C_i}{T} - y_t a_{i,j(t)} \big),
\end{equation}
where $y_t$ denotes the indicator that represents the accept/reject action for the customer in period $t$. To be more precise, we have 
\[y_t = \mathbf{1}_{r_{j(t)}>\sum_{i=1}^m \theta_i^{(t)} a_{i,j(t)}}.\]

Note that the definition of $g_t(\cdot)$ is related to the second term of the objective in \eqref{eq:dualLP}. In our analysis, we will use the OGD property to prove bounds on the sum of $g_t(\cdot)$, which then helps to bound the Lagrangian relaxation \eqref{eq:dualLP}.



\subsection{Primal-dual algorithm}
We provide in Algorithm \ref{alg:ff} the detailed procedure of our first primal-dual algorithm. Algorithm \ref{alg:ff} takes as input a start time $t_1 \in \{1,2,\ldots, T\}$, an end time $t_2 \in \{t_1, t_1+1,\ldots,T\}$, and a vector of capacities $B$ at time $t_1$. These input values are used by the algorithm in the next section, which uses Algorithm \ref{alg:ff} as a subroutine. In particular, the end-of-execution time $t_2$ can be earlier than the end of the  horizon $T$. 


For convenience, we denote $L=T-t_1+1$ as the length of the remaining time horizon that scales with $B$. With capacity $B$ and time length $L$, the definition of $g_t(\cdot)$ \eqref{eq:gtthetaC} can be written as
\begin{equation} \label{eq:gtthetaBL}
    g_t(\theta^{(t)}) = \sum_{i=1}^m \theta_i^{(t)} \big( \frac{B_i}{L} - y_t a_{i,j(t)} \big).
\end{equation}


For the OGD steps in Algorithm~\ref{alg:ff}, we set $[0, \bar{\theta}]^m$ to be the convex feasible set of the bid prices (i.e., the convex set $\cal{K}$ used in Algorithm~\ref{alg:ogd}), where $\bar \theta$ is defined as follows.
We first define $ \bar{\alpha_i} := \max_{j\in[n]: a_{ij} \neq 0} \frac{r_j}{a_{ij}}$, which is an upper bound on the revenue that can be achieved from one unit of resource $i$. Let $B_{\max}:=\max_{i\in[m]} B_i$ and $B_{\min}:=\min_{i\in[m]} B_i$ be the maximum and minimum resource capacity, respectively. Then we define $\bar \theta$ as
\begin{equation} \label{eq:alpha}
    \bar{\theta}:=\frac{B_{\max}}{B_{\min}} \sum_{i=1}^m \bar{\alpha_i}.
\end{equation} 

Throughout all of our algorithms, we will ensure that $B_\text{min}>0$.
The definition of $\bar \theta$ is constructed so that $\bar \theta$ serves as an upper bound on the optimal dual variables, which we formally state in  Lemma~\ref{lm:alpha} in the appendix.   
 
In order to apply Proposition \ref{prop:ogd}, we calculate $D$ and $G$ as follows. For any two vectors of bid prices $\theta,\theta' \in [0, \bar{\theta}]^m$, by the Cauchy-Schwarz inequality, we have 
 $\| \theta - \theta'\| \leq \bar{\theta} \sqrt{m}.$
Thus, we use 
\begin{equation} \label{eq:ogd_para_D}
    D = \bar{\theta} \sqrt{m}
\end{equation} as the upper bound on the diameter of region $[0, \bar{\theta}]^m$.

Recall that $G$ is the upper bound of the gradient of function $g_t(\theta)$. Let 
\begin{equation} \label{eq:bara}
\bar{a} = \max_{i\in[m],j\in[n]} a_{ij}
\end{equation}
denote the maximum consumption of any resource from any customer type. Starting from the definition of $g_t(\theta)$ \eqref{eq:gtthetaBL}, we have
\begin{align*}
\Vert \nabla g_t(\theta) \Vert &= \Vert \nabla \sum_{i=1}^m \theta_i \big( \frac{B_i}{L} - y_t a_{i,j(t)} \big) \Vert \\
&=\left\Vert
    \frac{B}{L} - y_t A_{j(t)} \right\Vert
\\&\leq \left\Vert \frac{B}{L} \right\Vert +  \Vert y_t A_{j(t)} \Vert
\\&\leq \frac{B_{\max}}{L}\sqrt{m} + \bar{a}\sqrt{m},
\end{align*}
where the last two inequalities follow from the triangle inequality and the Cauchy-Schwarz inequality, respectively.
Thus, we use 
\begin{equation}\label{eq:ogd_para}
G = \frac{B_{\max}}{L}\sqrt{m} + \bar{a}\sqrt{m}.
\end{equation}

We stress that $\bar \theta$, $D$ and $G$ depend on either $B_{\max}/B_{\min}$ or $B_{\max}/L$. Both of the ratios $B_{\max}/B_{\min}$ and $B_{\max}/L$ become constants when we execute Algorithm~\ref{alg:ff} alone over the entire horizon (i.e., when $B=C$ and $L=T$),  because the values of $C$ and $T$ are proportional to the system size (see Section~\ref{sec:asymptotic}). 

Algorithm~\ref{alg:ff} maintains a vector of remaining resources. It terminates at the end of period $t_2$, or when any resource is not sufficient at the beginning of a period.

\begin{algorithm}[h]
\caption{Primal-Dual Optimal Algorithm on the Fluid Scale ($\FF$)} \label{alg:ff}
\begin{algorithmic}[1]
\footnotesize
\STATE \textbf{Input:} start time $t_1$, end time $t_2$, initial capacity at start time $B \in \R^m$;
\STATE \textbf{Initialize:} $\theta^{(t_1)} \gets 0 \in \real^m$, $B(t_1) \gets B$, $L \gets T-t_1+1$;\\ \hspace{3.5em} 
$B_{\max}=\max_{i\in[m]} B_i$,
$B_{\min}=\min_{i\in[m]} B_i$,
$ \bar{\alpha_i} = \max_{j\in[n]: a_{ij} \neq 0} \frac{r_j}{a_{ij}}$,
$\bar{\theta}=\frac{B_{\max}}{B_{\min}} \sum_{i \in [m]} \bar{\alpha_i}$; \\ \hspace{3.5em} 
$\bar{a} = \max_{i\in[m],j\in[n]} a_{ij}$,
$G \gets \frac{B_{\max}}{L} + \sqrt{m} \bar{a}$, 
$D \gets \bar{\theta} \sqrt{m} $.

\FOR{$t =t_1,t_1+1,\ldots,t_2$}
\STATE Observe customer of type $j(t)$. Set $y_t \gets \mathbf{1} \big( r_{j(t)} > \sum_{i=1}^m \theta_i^{(t)} a_{i,j(t)} \big)$.
\IF{$A_j \leq B(t)$ for all $j\in [n]$}  
\IF{$y_t$ equals to $1$}  
\STATE Accept the customer.
\ELSE
\STATE Reject the customer.
\ENDIF
\STATE Set $B(t+1) \gets B(t) - y_t A_{j(t)}$.
\ELSE
\STATE Break 
\ENDIF
\STATE  Construct function
\[ g_t(\theta) = \sum_{i=1}^m \theta_i \left( \frac{B_i}{L} - y_t a_{i,j(t)} \right) .\]
\STATE Update the dual variables using the OGD procedure (Step 5 of Algorithm~\ref{alg:ogd})
\[ \eta_t \gets \frac{D}{ G\sqrt{t-t_1 + 1} },\]
\[ \theta^{(t+1)} \gets
\theta^{(t)} - \eta_t \nabla_\theta g_t(\theta^{(t)}),\]
\[ \theta_i^{(t+1)}  \gets \min \bigg(\max \big(0, \theta_i^{(t+1)} \big) , \; \bar{\theta} \bigg)
\text{ for all } i \in [m] \]
\ENDFOR
\end{algorithmic}
\end{algorithm}

\subsection{Regret analysis}

In this section, we sketch the key ideas in proving the regret bound of Algorithm~\ref{alg:ff}. For ease of presentation, we use $t_1 = 1$, $t_2 = T$ and $B = C$ in this section, and defer rigorous proofs to the appendix. 

Let $V^{\FF}$ denote the  revenue of Algorithm~\ref{alg:ff} from the $T$ periods. The regret of Algorithm~\ref{alg:ff} can be upper-bounded as
\[ \bE[V^{\HO} - V^{\FF}] =
\bE[V^{\HO}] - V^{\DLP} + V^{\DLP} - \bE[V^{\FF}] \leq V^{\DLP} - \bE[V^{\FF}],\]
where the inequality is due to \eqref{eq:vho_vdlp}.
It thus suffices to prove an upper bound on $V^{\DLP} - \bE[V^\FF]$. 

 We present in Proposition~\ref{prop:ff:regret} a high probability bound for the gap $V^{\DLP} - V^\FF$. 

\begin{proposition} \label{prop:ff:regret}
With probability at least $1-\frac{1}{T}$, we have
\begin{equation}
     V^{\DLP} - V^{\FF} \leq D_2 \sqrt{T} (\log T)^\frac{1}{2} + D_1 \sqrt{T} + D_0,
\end{equation}
where $D_2$, $D_1$ and $D_0$ do not depend on $T$.
\end{proposition}

This bound is more informative than many regret bounds in the literature that are shown in expectations, e.g., in \cite{agrawal2016linear}.
We need this high probability bound to prove further properties in the next section.
Similar high probability bounds are also derived in
\cite{agrawal2014fast}.




Our analysis is based on an elegant construction of a martingale that incorporates the Lagrangian relaxation \eqref{eq:dualLP} and the cost function \eqref{eq:gtthetaC} in the OGD procedure.
 Specifically, we show  
 that the stochastic process
\[
M_t := \sum_{s=1}^t \left[ \frac{\LP(\theta^{(s)})}{T} - y_s r_{j(s)} - g_s(\theta^{(s)})\right]
\]
is a martingale, and apply  Azuma's inequality to obtain a high-probability upper bound of $M_t$.

The intuition behind the construction of $M_t$ is as follows:
\begin{itemize}
    \item The first term in the brackets $\LP(\theta^{(s)}) / T$ is related to the DLP by weak duality $\LP(\theta^{(s)}) \geq V^\DLP$.
    \item The second term $y_s r_{j(s)}$ is the revenue of Algorithm~\ref{alg:ff} in period $s$.
    \item For the sum of $g_s(\theta^{(s)})$, we apply Proposition~\ref{prop:ogd}, as well as the trick in \citet{agrawal2014fast}, to upper-bound the loss from OGD and any loss due to lack of resources.
\end{itemize}




The regret bound of Algorithm~\ref{alg:ff} immediately follows from Proposition~\ref{prop:ff:regret}.

\begin{theorem}\label{thm:regret-ff}
The regret bound of Algorithm~\ref{alg:ff} is $O\left( \sqrt{T}(\log T)^{\frac{1}{2}}\right)$.
\end{theorem}  
\begin{proof}{Proof.}
\begin{align}
\bE[V^{\HO} - V^{\FF}] &\leq \nonumber \bE[V^{\DLP} - V^{\FF}]
\\&\leq \frac{1}{T} V^{\DLP} + \left(1-\frac{1}{T} \right)
\left( D_2 \sqrt{T}(\log T)^\frac{1}{2} + D_1 \sqrt{T} + D_0 \right) \nonumber\\
&\leq \frac{1}{T} \left( T \sum_{j=1}^n r_j \lambda_j\right)
+ \left(1-\frac{1}{T} \right) \left( D_2 \sqrt{T}(\log T)^\frac{1}{2} + D_1 \sqrt{T} + D_0 \right) \nonumber \\
&= O\left( \sqrt{T}(\log T)^{\frac{1}{2}}\right).\label{eq:regret:ff}
\end{align}
Above, the first inequality is from Equation \eqref{eq:vho_vdlp}, and the second inequality is from Proposition \ref{prop:ff:regret}. If the high probability event in Proposition \ref{prop:ff:regret} happens, we have $V^{\DLP} - V^{\FF} \leq D_2 \sqrt{T}(\log T)^\frac{1}{2} + D_1 \sqrt{T} + D_0$. Otherwise, we have $V^{\DLP} - V^{\FF} \leq V^{\DLP}$. 
\halmos
\end{proof}

Therefore, Algorithm~\ref{alg:ff} is optimal on the fluid scale.

\section{Optimal Algorithm on the Diffusion Scale}\label{sec:fd}
In this section, we present Algorithm~\ref{alg:fd} that is optimal on the diffusion scale. 
The algorithm calls Algorithm~\ref{alg:ff} as subroutines and furthermore applies a thresholding technique that divides customers into three classes. Algorithm~\ref{alg:fd} then uses a different online decision rule for each customer class.

Compared to the thresholding algorithms in \citet{bumpensanti2020re}, the design of the threshold values in Algorithm~\ref{alg:fd} addresses two additional challenges. First, since DLP solutions are not available, we device a statistical method to calculate the threshold values based on past online information. Second, while the threshold values are based on primal information (i.e., the number of customers accepted), the bid-price vector is a dual solution. In the presence of degeneracy, the primal and dual information may not be complementary. To overcome this difficulty, we make a mild assumption on the DLP's degeneracy, and prove in the next section a proposition that guarantees that the DLP has a unique optimal primal solution (but possibly multiple optimal dual solutions).

\subsection{An Additional High Probability Bound for Algorithm~\ref{alg:ff}}

In this section, we present another high probability bound in  Proposition\ref{prop:ff:x} for Algorithm~\ref{alg:ff}. This bound results from Lemma~\ref{lm:asmp:1} and guarantees that the allocation made by Algorithm~\ref{alg:ff} is close to the unique optimal primal solution of the DLP.

We show Proposition~\ref{prop:ff:x} based on a mild assumption, which is the same as assumption 3 in \cite{agrawal2014dynamic}.
Consider the following LP in the standard form. Given capacity vector $b \in \real^{m+}$, upper bound vector $v=[v_1, \ldots, v_n]^\top$,
decision variables $z = [z_1, \ldots, z_n]^\top$ and $\varepsilon = [\varepsilon_1, \ldots,\varepsilon_m ]^\top$,
we define             
\begin{align}\label{eq:lp_b_u}
\begin{split}
  \LP(b,v) := \max_{z,\varepsilon} \;\;& \sum_{j=1}^n r_j z_j  \\
  \text{s.t. } & \sum_{j=1}^n A_j z_j + \varepsilon = b \\
  & 0 \leq z_j  \leq u_j,\;\; \forall j \in [n]\\
  & \varepsilon_i \geq 0, \;\; \forall i \in [m] .
\end{split}
\end{align}

Note that the hindsight LP shown in \eqref{eq:HO} and the DLP shown in \eqref{eq:DLP} are both special instances of $\LP(b,v)$. Specifically, when $b=C$ and $v_j = \lambda_j T$ for $j\in[n]$, $\LP(b,v)$ is equivalent to \eqref{eq:DLP}.

Let $\bar{A} := [A_1, \ldots, A_n, \mathbf{e}_1, \ldots, \mathbf{e}_m] \in \real^{m\times(m+n)}$ denote the columns of the capacity constraints in $\LP(b,v)$, and $\bar{r} := [r_1,\ldots,r_n, 0, \ldots, 0]^\top \in \real^{m+n}$ the coefficients in the objective.

\begin{assumption}\label{asmp:1}
The problem inputs in \eqref{eq:lp_b_u} are in general position. Namely, for any bid price vector $p \in \real^m$, there can be at most $m$ columns among $l\in[m+n]$ such that $p^\top \bar{A}_l = \bar{r}_l^\top$.
\end{assumption}

\begin{remark} Assumption~\ref{asmp:1} is not necessarily true for all inputs. However, as discussed in \cite{agrawal2014dynamic} and \cite{devanur2009adwords}, one can always randomly perturb the revenue vector $r$ by adding to $r_j$ a random variable $\xi_j$ that follows a uniform distribution over a small interval $[0,\eta]$. In this way, with probability one, no bid-price vector $p$ can satisfy $m+1$ equations simultaneously among $p^\top \bar{A}_l = \bar{r}_l^\top$. The effect of this perturbation on the objective can be made arbitrarily small. 
\end{remark}

Assumption~\ref{asmp:1} implies the following lemma, proved in Appendix~\ref{ap:ff}.

\begin{lemma}\label{lm:asmp:1}
Let $\tilde w^*=[z^*,\varepsilon^*]$ be an optimal primal solution to $\LP(b,v)$ as shown in \eqref{eq:lp_b_u}. 
Under Assumption~\ref{asmp:1}, $\tilde w^*$ is unique. 
In addition, given any feasible solution $\tilde w=[z,\varepsilon]$ of \eqref{eq:lp_b_u} and direction $d:= \frac{\tilde w- \tilde w^*}{\Vert \tilde w - \tilde w^* \Vert} \in \real^{m+n}$, $\bar{r}^\top d$ is guaranteed to be upper-bounded by a constant.
\end{lemma}

Recall $w^*$ is an optimal solution to DLP \eqref{eq:DLP}. Let $x_j^{\FF} $ be the number of type-$j$ customers accepted by  Algorithm~\ref{alg:ff} with start time $t_1=1$, end time $t_2=T$ and initial capacity $B=C$. 
By Lemma~\ref{lm:asmp:1}, we show the following proposition.



\begin{proposition} \label{prop:ff:x}
Under Assumption~\ref{asmp:1}, with probability at least $1-O(\frac{1}{T})$, we have
\begin{equation}\label{eq:ff:x}
    |x^{\FF}_j - w_j^*| \leq \rho \sqrt{T}(\log T)^\frac{1}{2} 
\end{equation}
for all $j\in[n]$, where $\rho$ does not depend on $T$.
\end{proposition}


The proof of Proposition~\ref{prop:ff:x} is based on geometric analysis. The detailed proof is in Appendix~\ref{ap:ff}.

\subsection{Algorithm}

We present Algorithm~\ref{alg:fd} with three technical parameters $\alpha$, $\beta$ and $\gamma$ that satisfy $0 < \alpha < \frac{1}{2}$, $\frac{\alpha}{2} \leq \beta < \frac{1}{2}$ and $ 0 < \gamma < \frac{\alpha}{2}$. In the technical theorems, we will choose special values of $\alpha$, $\beta$ and $\gamma$ for proving regret bounds.

The algorithm defines two special time points $\lceil T^a \rceil$ and $\lceil T-T^b \rceil$ that divide the horizon into three intervals. Here, the values of $a$ and $b$ are defined by the input parameters $\alpha$ and $\beta$. 

For instance, if we run Algorithm~\ref{alg:fd} using $\alpha = 1/3$ and $\beta = 1/3$, then we have $T^a = T^{1/3}$, $T^b= T^{5/6}$ (by Line 2 of Algorithm~\ref{alg:fd}), and Algorithm~\ref{alg:fd} divides the horizon into three intervals: (i) from period $1$ to period $\lceil T^{1/3} \rceil$, (ii) from period $\lceil T^{1/3} \rceil+1$ to period $\lceil T- T^{5/6} \rceil$ and (iii) from period $\lceil T-T^{5/6} \rceil+1$ to period $T$.




Algorithm~\ref{alg:fd} uses a different strategy in each of the three intervals, and thus we call each of the three intervals as a different ``phase'' of Algorithm~\ref{alg:fd}.

\begin{itemize}
    \item In phase I, Algorithm~\ref{alg:fd} runs Algorithm~\ref{alg:ff} as a subroutine. Let $x_j^\text{I}$ denote the number of type-$j$ customers accepted in phase I. At the end of phase I, customer types are divided into three classes based on $x_j^\text{I}$. In particular, if $x_j^\text{I}$ exceeds a certain threshold, then we add type $j$ into an ``accept'' class $\calA$, or if $x_j^\text{I}$ is below another threshold, then we add type $j$ into a ``reject'' class $\calR$. The thresholds are given by the technical parameters $\alpha$ and $\gamma$.

    The purpose of thresholding is to exclude customer types for which concentration inequalities are not effective. Intuitively, when very few customers are accepted (i.e., $x_j^\text{I}$ is small), there is very large (relative) estimation error on the chance of accepting a customer. Similarly, when very few customers are rejected (i.e., $x_j^\text{I}$ is large), there is very large (relative) estimation error on the chance of rejecting a customer. When we are estimating the DLP solution, the thresholding technique helps us focus on the customer types for which concentration inequalities lead to small relative uncertainty intervals. 


    

\item At the beginning of phase II, Algorithm~\ref{alg:fd} sets initial capacity $B(t_1+l_1) = \frac{T-l_1}{T}C$ where $l_1$ denotes the time length of phase I, and $B(t) = [B_1(t), \ldots, B_m(t)]^\top$ denotes the vector of remaining resources at the beginning of period $t$. Algorithm~\ref{alg:fd} also creates a \emph{virtual} copy of vector $B(t)$, denoted as $B'(t)$, and initiates a subroutine of Algorithm~\ref{alg:ff} that updates $B'(t)$.


In each period $t$ of phase II, if both $B(t)$ and $B'(t)$ are sufficient, i.e., $B_i(t) \geq a_{ij}$ and $B_i'(t) \geq a_{ij}$ for all $i\in[m]$ and $j\in[n]$, then when a customer of type $j(t)$ arrives, the subroutine calculates $y_t$ according to the definition of Algorithm~\ref{alg:ff}, namely
\[
y_t = \mathbf{1}\Big(r_{j(t)} > \sum_{i=1}^m \theta^{(t)}_i a_{i,j(t)}\Big),
\]
and Algorithm~\ref{alg:fd} makes admission decision
\begin{align*}
z_t &= \left\{\begin{matrix}
\; 0, & \text{ if $j(t) \in \calR$ }, &\\  
\; 1, & \text{ if $j(t) \in \calA$ }, & \\ 
\; y_t, & \quad \text{ if $j(t) \notin(\calA \cup \calR) $}. &\\ 
\end{matrix}\right. 
\end{align*}
The algorithm then updates capacity vectors 
\[
B(t+1) = B(t) - z_t A_{j(t)},\;\; B'(t+1) = B'(t) - y_t A_{j(t)},
\]
and bid prices $\theta^{(t)}$ according to the subroutine based on $y_t$ (see Line 23 of Algorithm~\ref{alg:fd}).

If $B(t)$ is sufficient, but $B'(t)$ is not sufficient, i.e., there exists $j$ such that $A_j \leq B'(t)$ is violated, then the subroutine of Algorithm~\ref{alg:ff} stops, and Algorithm~\ref{alg:fd} makes admission decision:
\begin{align*}
z_t &= \left\{\begin{matrix}
\; 0, & \text{ if $j(t) \in \calR$ }, &\\  
\; 1, & \text{ if $j(t) \in \calA$ }, & \\ 
\; 0, & \quad \text{ if $j(t) \notin(\calA \cup \calR) $}, &\\ 
\end{matrix}\right. 
\end{align*}
and then updates $B(t+1) = B(t) - z_t A_{j(t)}$.


If $B(t)$ is not sufficient, phase II stops.

The reason we construct the virtual copy of capacity vector $B'(t)$ is because we can then easily apply previous results of Algorithm~\ref{alg:ff} to the decisions $\{y_t\}_t$ made by the subroutine. Notice that the difference between $y_t$ and the real decision $z_t$ is only caused by the customer types in $\cal{R}$ and $\cal{A}$. We will separately analyze those customer types in order to characterize $z_t$.

\item In phase III, Algorithm~\ref{alg:fd} runs Algorithm~\ref{alg:ff} over the rest of the time horizon. For this phase, the main challenge is that the remaining capacity values are no longer proportional to the number of time periods left. As a result, we cannot simply use Proposition~\ref{prop:ff:regret} (which assumes that both the capacity values and the length of the time horizon are proportional to the system size) to obtain the regret bound for phase III. To overcome this difficulty, we show a more general capacity-dependent regret bound for Algorithm~\ref{alg:ff}. In addition, we create a new virtual copy of capacity vector $B''$ for phase III, in which the ratio between the maximum and minimum resource capacities is related to the length of phase III. 
\\
Specifically, for each resource $i$, we set the initial virtual capacity for phase III to 
\[ \max\{ T^{3b/4}, 
\min\{B_i(1+l_1+l_2), \bar{a} T^b\} \},\]
where $B_i(1+l_1+l_2)$ is the real remaining capacity at the beginning of phase III and $T^b$ is the length of phase III. Note that the upper bound $\bar{a} T^b$ is essentially the maximum amount of any resource that can be allocated in phase III.
\\
The subroutine of Algorithm~\ref{alg:ff} runs on the new virtual capacity values. Algorithm\ref{alg:fd} copies the decisions of the subroutine whenever the decision is feasible (i.e., when $B(\cdot)$ is enough for the decision). 
\end{itemize}

\begin{algorithm}[h]
\caption{Primal-Dual Optimal Algorithm on the Diffusion Scale ($\FD$)} \label{alg:fd}
\begin{algorithmic}[1]
\scriptsize
\STATE \textbf{Input:} 
parameters $\alpha$, $\beta$ and $\gamma$ where $0 < \alpha < \frac{1}{2}$,
$\frac{\alpha}{2} \leq \beta < \frac{1}{2}$ and
$ 0 < \gamma < \frac{\alpha}{2}$.

\STATE \textbf{Initialize:}
set phase parameters 
$a = \alpha$, $b=\frac{1}{2}+\beta$, thresholding parameter $c = \frac{1}{2} \alpha + \gamma$, \\
\hspace{1cm} 
set $l_1 \gets \lceil T^{a} \rceil$,
$l_2 \gets  \lceil T - T^{a} - T^{b} \rceil$ ,
$B(t_1) = C$,
set $\bar{a} = \max_{i\in[m],j\in[n]} a_{ij}$,
set $\calA = \{ \; \}$ and $\calR = \{ \; \}$. \\
\hspace{1cm}  (Remark: the algorithm runs all the three phases when $t_2 > t_1 + l_1 + l_2$.)

\STATE \textbf{Phase I.} 
Run Algorithm~\ref{alg:ff} with start time $1$, end time $l_1$, and initial capacity $C\cdot(l_1/L)$.
\STATE \textbf{Phase II.} Thresholding: 
let $x_j$ be the number of type-$j$ customers accepted in phase I.
\FOR{$j \in [n]$}
\IF{$x_{j} < \lambda_{j} T^{c}$}  
\STATE add $j$ to list $\calR$
\ELSIF{$x_{j} > \lambda_{j} (T^{a} - T^{c})$}
\STATE add $j$ to list $\calA$
\ENDIF
\ENDFOR
\STATE Set $B(1+l_1)= C \cdot (T-l_1)/T$, $B'(1+l_1)= C\cdot (T-l_1)/T$.
\STATE Run Algorithm~\ref{alg:ff} Line 2 with start time $1+l_1$, end time $l_1+l_2$, and initial capacity $ C\cdot (T-l_1)/T$.
\FOR{$t = 1+l_1, \ldots, l_1+l_2$}
\STATE Observe customer of type $j(t)$. Set $y_t \gets \mathbf{1} \big( r_{j(t)} > \sum_{i=1}^m \theta_i^{(t)} a_{i,j(t)} \big)$
\IF{$A_j \leq B(t)$ for all $j\in [n]$}  
\IF{$j(t) \in \calR$}  
\STATE Reject the customer. Set $z_t \gets 0$.
\ELSIF{$j(t) \in \calA$}
\STATE Accept the customer. Set $z_t \gets 1$.
\ELSIF{$A_j \leq B'(t)$ for all $j\in [n]$}
\STATE Set $z_t \gets y_t$.
\STATE Set $B'(t+1) \gets B'(t) - y_t A_{j(t)}$. \STATE Run Algorithm~\ref{alg:ff} Line 15 and Line 16. 
\ELSE{}
\STATE Set $y_t \gets 0$, $z_t \gets 0$. 
\ENDIF
\STATE 
Set $B(t+1) \gets B(t) - z_t A_{j(t)}$.
\IF{$z_t = 1$}
\STATE Accept the customer.
\ELSE{}
\STATE Reject the customer.
\ENDIF
\ELSE
\STATE Break
\ENDIF
\ENDFOR
\STATE \textbf{Phase III.}

Set $B''(1+l_1+l_2)=\max\{ T^{3b/4}, 
\min\{B_i(1+l_1+l_2), \bar{a} T^b\}\}$.
\STATE Run Algorithm~\ref{alg:ff} Line 2 with start time $1+l_1+l_2$, end time $T$ and initial capacity $B''(1+l_1+l_2)$
\FOR{$t = 1+l_1+l_2, \ldots, T$}
\STATE Observe customer type $j(t)$. Set $y_t \gets \mathbf{1} \big( r_{j(t)} > \sum_{i=1}^m \theta_i^{(t)} a_{i,j(t)} \big)$
\IF{$A_j \leq B''(t)$ for all $j\in [n]$}
\IF{$A_{j(t)} \leq B(t)$}  
\STATE Set $z_t \gets y_t$.
\ELSE{}
\STATE Set $z_t \gets 0$. 
\ENDIF
\STATE Set $B''(t+1) \gets B''(t) - y_t A_{j(t)}$. \STATE Run Algorithm~\ref{alg:ff} Line 15 and Line 16. 
\ELSE
\STATE $z_t\gets 0$
\ENDIF
\STATE 
Set $B(t+1) \gets B(t) - z_t A_{j(t)}$.
\IF{$z_t = 1$}
\STATE Accept the customer.
\ELSE{}
\STATE Reject the customer.
\ENDIF
\ENDFOR
\end{algorithmic}
\end{algorithm}



\subsection{Regret analysis}

In the following theorem, we provide the regret bound of Algorithm~\ref{alg:fd} for the original NRM problem. The detailed analysis is provided in Appendix~\ref{ap:fd}.

\begin{theorem}\label{thm:fd:regret} 
The regret bound of Algorithm~\ref{alg:fd} is
\[
O(T^{\alpha})+O(Te^{-T^{\varepsilon}})+O(T^{\frac{3}{8}+\frac{3}{4}\beta + \frac{1}{2}\varepsilon} ),
\]
where the parameters $\alpha$, $\beta$, $\gamma$ satisfy the condition $0 < \alpha < \frac{1}{2}$, $\frac{\alpha}{2} \leq \beta < \frac{1}{2}$, $ 0 < \gamma < \frac{\alpha}{2}$, and $\varepsilon = \frac{\log(\log T)}{ \log T}$

If we let the parameters be
\[
\alpha= \frac{3 \log(\log T)}{ 2 \log T}, \;
\beta = \frac{\log(\log T)}{ \log T}, \;
\gamma =  \frac{2\log(\log T)}{3 \log T} .
\]
The algorithm is optimal on the diffusion scale. Specifically, we have
\[ \bE[V^{\HO} - V^{\FD}] = O\left(T^{\frac{3}{8}}( \log T)^{\frac{5}{4}} \right) . \]
\end{theorem}

We outline the key idea of the proof here. Let HO$(t)$ denote the hindsight LP for periods from $t$ to $T$. Formally, HO$(t)$ is given by
\begin{align}\label{eq:HO_t}
\begin{split}
V^{\text{HO}}(t) := \max_z & \sum_{j=1}^n r_j z_j \\
\text{s.t.} \; & \sum_{j=1}^n A_j z_j \leq B(t),\\
& 0 \leq z_j \leq \Lambda(t,T), \; \forall j \in [n].
\end{split}
\end{align}
Let $\bar{z}_j(t)$ for $j\in[n]$ denote an optimal solution to $\HO(t)$. We have $V^{\HO}(t) = \sum_{j\in[n]} r_j \bar{z}_j(t) $.

Let $t_1':= l_1+1$ be the starting time period of phase II, $t_2'=l_1+l_2+1$ the starting time period of phase III, where $l_1=\lceil T^a \rceil$, $l_2=\lceil T-T^a-T^b \rceil$ are the lengths of phase I and phase II, respectively.

To analyze the performance of  Algorithm~\ref{alg:fd}, we define an auxiliary algorithm $\AuxHO$-I that follows Algorithm~\ref{alg:fd} in phase I, and then uses the hindsight optimal decisions, namely the optimal solution $\bar{z}(t_1')$ to $\HO(t_1')$, in the remaining time periods. Similarly, we define an auxiliary algorithm $\AuxHO$-II that follows Algorithm~\ref{alg:fd} in phase I and phase II, and then uses the hindsight optimal decisions, namely, the optimal solution $\bar{z}(t_2')$ to $\HO(t_2')$, in the remaining time periods.
Let $V^{\AuxHO\text{-}\I}$, $V^{\AuxHO\text{-}\II}$ denote the revenue of the auxiliary algorithms $\AuxHO$-I and $\AuxHO$-II, respectively.
By definition, we decompose the regret of Algorithm~\ref{alg:fd} as
\begin{align}\label{eq:decompose1}
    \bE[V^{\HO} - V^{\FD}] &=
    \bE[V^{\HO} - V^{\AuxHO\text{-I}}] 
    + \bE[V^{\AuxHO\text{-I}} - {V}^{\AuxHO\text{-II}}]
    + \bE[{V}^{\AuxHO\text{-II}} - V^{\FD}].
\end{align}
Our proof of Theorem~\ref{thm:fd:regret} depends on the following proposition, which is proved in Appendix~\ref{ap:fd}.

\begin{proposition} \label{prop:fd:regret}
Consider Algorithm~\ref{alg:fd} with parameters $\alpha, \beta, \gamma$ that satisfy $0<\alpha < \frac{1}{2}$,
$\frac{\alpha}{2} \leq \beta < \frac{1}{2}$ and $0 < \gamma < \frac{\alpha}{2}$.
The algorithm calculates $a = \alpha$, $b = \frac{1}{2}+\beta$, $c = \frac{\alpha}{2} + \gamma$.

Under Assumption~\ref{asmp:1}, given $\varepsilon > 0$ such that $ 0 < \varepsilon < 2\gamma$, $\varepsilon\leq \beta$ and $\varepsilon \geq \frac{\log (\log T)}{\log T}$, we have
\begin{align*}
    (1). & \;\; \bE[V^{\HO} - V^{\AuxHO\text{-\normalfont{I}}}] = O(T^a) ;\\
    (2). & \;\; \bE[V^{\AuxHO\text{-\normalfont{I}}} - {V}^{\AuxHO\text{-\normalfont{II}}}]
    = O( T e^{-T^{\varepsilon}}) ;\\
    (3). & \;\;  \bE[V^{\AuxHO\text{-\normalfont{II}}} - V^{\FD}]
    = O\big(T^{\frac{3}{4}b + \frac{\varepsilon}{2}} \big).
\end{align*}
\end{proposition}

In fact, each item in Proposition~\ref{prop:fd:regret} describes a regret bound of Algorithm~\ref{alg:fd} relative to the hindsight optimum $\HO$, $\HO(t_1')$ and $\HO(t_2')$ in each phase. Intuitively, for phase I, $V^{\HO} - V^{\AuxHO\text{-\normalfont{I}}} $ is upper bounded by $V^{\HO} - V^{\HO}(t_1')$, which is of order $O(l_1)=O(T^a)$.

Let $\DLP(t)$ denote the DLP for periods from $t$ to $T$, given by 
\begin{align}\label{eq:DLP_t}
\begin{split}
V^{\text{DLP}}(t) = \max_x\; & (T-t+1) \sum_{j=1}^n r_j x_j \\
\text{s.t.} \; & \sum_{j=1}^n A_j x_j \leq \frac{B(t)}{T-t+1},\\
& 0 \leq x_j \leq \lambda_j, \; \forall j \in [n].
\end{split}
\end{align}

Let $x_j^*(t)$ for $j\in[n]$ denote an optimal solution to $\DLP(t)$, and $x_j^*$ for $j\in[n]$ the optimal solution to the $\DLP$ in \eqref{eq:DLP1}. 
By definition of Algorithm~\ref{alg:fd} (Line 12), we have $B(t_1')/(T-l_1)=C/T$, and hence $x_j^*(t_1')$ is also an optimal solution to \eqref{eq:DLP1}. Under Assumption~\ref{asmp:1}, we know by Lemma~\ref{lm:asmp:1} that $x_j^*(t)$ is unique (since \eqref{eq:DLP1} is a special case of $\LP(b,v)$). Thus, we have $x_j^*(t_1') = x_j^*$ for all $j\in[n]$.


Let $x_j^{\FD}(s_1,s_2)$ denote the number of type $j$ customers accepted by Algorithm~\ref{alg:fd} during periods $ s_1,\ldots, s_2$.
By Proposition~\ref{prop:ff:x}, we know $x_j^{\FD}(t_1', t_2'-1)$ for $j\notin(\calA \cup \calR) $ is close to $l_2\, x_j^*(t_1')= l_2 x_j^*$ with high probability, since the admission decision for customer types $j\notin(\calA \cup \calR)$ follows that of Algorithm~\ref{alg:ff}. 

In addition, we know $x_j^{\FD}(t_1',t_2'-1)=0$ for $j\in \calR$, and $x_j^{\FD}(t_1',t_2'-1)=1$ for $j\in \calA$.
We show in Lemma~\ref{lm:fd:p2:x} that the hindsight optimal solution $\bar{z}_j(t_1')$ is close to $(T-l_1) x_j^*$ with high probability. Combining these results with the thresholding conditions with respect to $x_j^{\FD}(1, t_1'-1)$, we show in Lemma~\ref{lm:fd:p2:regret} that $x_j^{\FD}(t_1',t_2'-1)$ satisfies
\begin{equation}
x_j^{\FD}(t_1',t_2'-1) \leq 
\bar{z}_j(t_1') \leq
x_j^{\FD}(t_1',t_2'-1) + \Lambda(t_2',T)
\end{equation}
for all customer types with high probability. This implies that the decision-maker is able to achieve the hindsight optimum of $\HO(t_1')$ if she follows the decision of Algorithm~\ref{alg:fd} during period $t_1',\ldots,t_2'-1$, and then accepts $\bar{z}_j(t_1') - x_j^{\FD}(t_1',t_2'-1)$ type-$j$ customers during periods $t_2',\ldots,T$.
Thus, the regret for phase II is small.

The result for phase III follows from a more general version of Proposition~\ref{prop:ff:regret}. 
In particular, phase III corresponds to an NRM problem with initial capacity $B(t_2')$ and horizon length $T^b$. It differs from the original NRM model in that the capacity values are not proportional to the length of the time horizon. Thus, we must prove a capacity-dependent regret bound of Algorithm~\ref{alg:ff}. The result for phase III directly follows from that capacity-dependent regret bound.

Now we complete the proof of Theorem~\ref{thm:fd:regret} by setting the technical parameters in Algorithm~\ref{alg:fd} as follows:
$$\alpha= \frac{3 \log(\log T)}{ 2 \log T}, \;
\beta = \frac{\log(\log T)}{ \log T}, \;
\gamma =  \frac{2\log(\log T)}{3 \log T} .$$

Note that the parameters $\alpha$, $\beta$, $\gamma$ satisfy the condition $0 < \alpha < \frac{1}{2}$, $\frac{\alpha}{2} \leq \beta < \frac{1}{2}$ and $ 0 < \gamma < \frac{\alpha}{2}$ when $T$ is large enough, e.g., when $T \geq 100$. 
Given $\varepsilon = \frac{\log(\log T)}{ \log T}$, we apply Proposition~\ref{prop:fd:regret} and get 
\[O(T^a) = O(T^{\alpha}) = O((\log T)^{\frac{3}{2}}) , \]
\[O(Te^{-T^{\varepsilon}})=O(T\cdot \frac{1}{T}) = O(1) , \]
\[O(T^{\frac{3}{4}b + \frac{\varepsilon}{2}} ) = O(T^{\frac{3}{8}} T^{\frac{3\beta}{4}+\frac{\varepsilon}{2}}) = O(T^{\frac{3}{8}}( \log T)^{\frac{5}{4}}) . \]

Combining the calculations above with \eqref{eq:decompose1} and the results in Proposition~\ref{prop:fd:regret}, we know the regret bound of Algorithm~\ref{alg:fd} is $O(T^{\frac{3}{8}}( \log T)^{\frac{5}{4}}) $, which shows Algorithm~\ref{alg:fd} is optimal on the diffusion scale.

\section{Numerical Experiments}\label{sec:numerical}
In this section, we show the numerical performance of Algorithm~\ref{alg:ff} and Algorithm~\ref{alg:fd}. In addition, we introduce several practical algorithms based on the structure of Algorithm~\ref{alg:ff} and Algorithm~\ref{alg:fd}. We first introduce the new algorithms. 

We propose a new algorithm (Algorithm~\ref{alg:fr}) that restarts Algorithm~\ref{alg:fd} at a sequence of specific time points of the horizon. As an extension of  Algorithm~\ref{alg:fr}, we propose a hybrid algorithm (Algorithm~\ref{alg:mr}) that is allowed to solve DLPs for a limited number of times. We analyze the performance of Algorithm~\ref{alg:mr} to demonstrate how solving DLPs helps further reduce the regret bound. The pseudocode and detailed description of Algorithm~\ref{alg:fr} and Algorithm~\ref{alg:mr} can be found in Appendix \ref{ap:resolve}.

We also consider variants of Algorithm~\ref{alg:fr} and Algorithm~\ref{alg:mr}. Both algorithms include a total of $S+1$ \emph{epochs}. 
Each epoch $u$, where $u=0,1,\ldots, S$, starts from period $T-\tau_{u}+1$ and ends at period  $T-\tau_{u+1}$, where $\tau_u$ is the time length between the starting period of epoch $u$ and period $T$. Next, we introduce Algorithm 7 and Algorithm 8 that \emph{warm-start} each epoch $u$ by reusing the bid prices of epoch $u-1$ for $1\leq u\leq S$. Algorithm 7 and Algorithm 8 are variants of Algorithm~\ref{alg:fr} and Algorithm~\ref{alg:mr}, respectively.

For Algorithm~\ref{alg:fd}, we use the input parameters as shown in Theorem~\ref{thm:fd:regret}. 
For Algorithm~\ref{alg:mr}, we experiment with different values of $U$, namely the number of epochs that require solving DLPs. 
For the remaining algorithms, we tune the input parameters to achieve a good algorithm performance.


\subsection{Single Resource}
We first consider an NRM problem with a single resource (i.e., $m=1$), and two types of customers (i.e., $n=2$). Let $k$ be the system size. In the experiment, we set horizon length $T=k$. 
The number of arrivals for each customer type follows an independent binomial distribution with mean $0.5 k$ given horizon length $k$ (i.e., $\lambda_1 = \lambda_2 = 0.5$).
Both types of customers consume one unit of the resource if accepted (i.e., $a_{ij}=1$ for $i=1$ and $j=\{1,2\}$). Customers of type $1$ generate revenue $r_1$ if accepted, and
customers of type $2$ generate revenue $r_2$ if accepted.

Following the setting in \citet{bumpensanti2020re}, we test two cases: 1) $r_1=2$, $r_2=1$ and 2) $r_1=5$, $r_2=1$. 
We experiment with capacities $C_i=0.7k, 0.8k, 0.9k$ for $i=\{1,2\}$, 
and horizon length $k=10^3, 2\times 10^3, \dots , 10^4$, for each case, respectively.

\begin{figure}[h]
\center
\includegraphics[height=6cm,width=8cm]{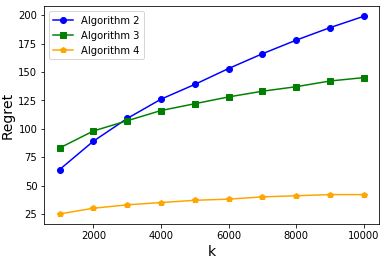}
\includegraphics[height=6cm,width=8cm]{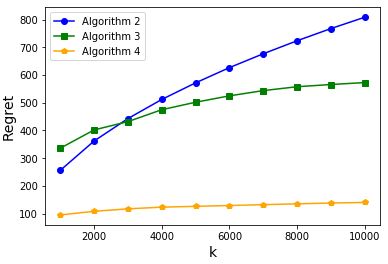}
\\
(a) $C=0.7k$, $r_1=2$ and $r_2=1$ \ \ \ \ \ \ \ \ \ \ \ \ \ \ \ \ \ \ \ (b)  $C=0.7k$, $r_1=5$ and $r_2=1$
\\
\includegraphics[height=6cm,width=8cm]{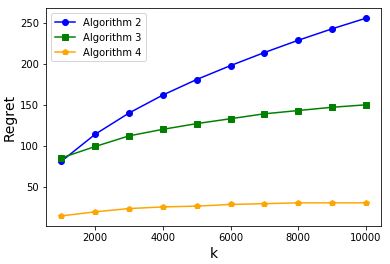}
\includegraphics[height=6cm,width=8cm]{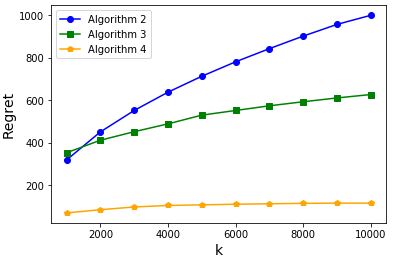}
\\
(c) $C=0.8k$, $r_1=2$ and $r_2=1$ \ \ \ \ \ \ \ \ \ \ \ \ \ \ \ \ \ \ \ (d)  $C=0.8k$, $r_1=5$ and $r_2=1$
\\
\includegraphics[height=6cm,width=8cm]{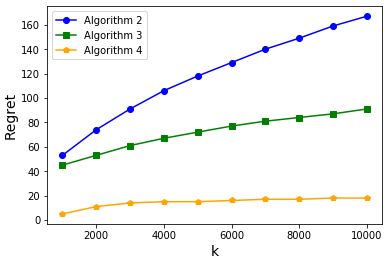}
\includegraphics[height=6cm,width=8cm]{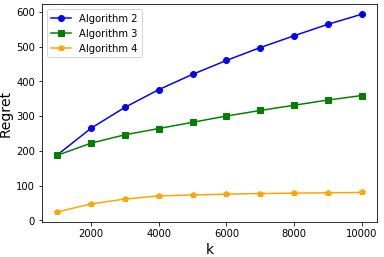}
\\
(e) $C=0.9k$, $r_1=2$ and $r_2=1$ \ \ \ \ \ \ \ \ \ \ \ \ \ \ \ \ \ \ \ (f)  $C=0.9k$, $r_1=5$ and $r_2=1$
\caption{Regret of Algorithm~\ref{alg:ff}, Algorithm~\ref{alg:fd}, and Algorithm~\ref{alg:fr} for $k=10^3, 2\times 10^3, \dots , 10^4$.}\label{fig:exp}
\end{figure}

Figure \ref{fig:exp} shows the regret of Algorithm~\ref{alg:ff}, Algorithm~\ref{alg:fd} and Algorithm~\ref{alg:fr}, namely the average gap between the algorithm's total revenue and the hindsight optimum, under different settings.
The first row shows the case where $C=0.7k$, $r_1=2$, $r_2=1$ and $C=0.7k$, $r_1=5$, $r_2=1$. The second row shows the case where $C=0.8k$, $r_1=2$, $r_2=1$ and $C=0.8k$, $r_1=5$, $r_2=1$. The third row shows the case where $C=0.9k$, $r_1=2$, $r_2=1$ and $C=0.9k$, $r_1=5$, $r_2=1$. 
We observe that in all three cases, algorithm Algorithm~\ref{alg:fr} performs the best among all three algorithms, and Algorithm~\ref{alg:ff} performs the worst. This comparison result is aligned with our theoretical regret bounds shown in \eqref{eq:regret:ff} and Theorem~\ref{thm:fd:regret}.


    

We next investigate the benefit of solving DLPs by examining the performance of Algorithm~\ref{alg:mr} under different LP-resolving times.
Recall that Algorithm~\ref{alg:mr} runs Algorithm~\ref{alg:lpt} as subroutines in the first $U$ epochs, and then runs Algorithm~\ref{alg:fd} as subroutines in the remaining epochs. We test the performance of Algorithm~\ref{alg:mr} for $U=0, 1, 2, 3, 4$ under two cases: 1) $r_1=2$, $r_2=1$, 2) $r_1=5$, $r_2=1$. We set capacity $C=0.8k$ and horizon length $k=10^3, 2\times 10^3, \dots , 10^4$.
In both cases, we observe that the regret decreases significantly if we resolves DLPs in the first four epochs. The benefit of resolving DLPs becomes less significant as the value of $U$ increases. 

\begin{figure}[h] 
\center
\includegraphics[height=6cm,width=8cm]{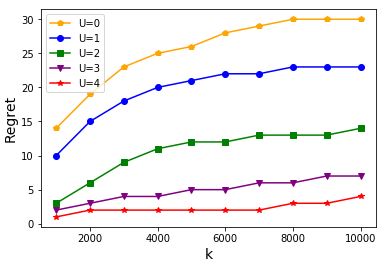}
\includegraphics[height=6cm,width=8cm]{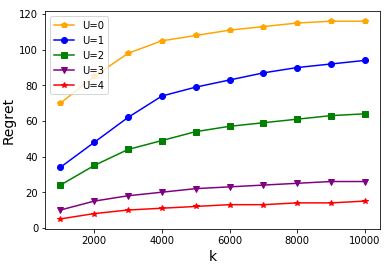}
\\
(a) $C=0.8k$, $r_1=2$ and $r_2=1$ \ \ \ \ \ \ \ \ \ \ \ \ \ \ \ \ \ \ \ (b) $C=0.8k$, $r_1=5$ and $r_2=1$
\caption{Regret of Algorithm~\ref{alg:mr} with parameter $U=1,2,3,4$ for $k=1000, 2000, \dots , 10000$.}
\label{fig:mr}
\end{figure}

Next, we test the performance of Algorithm 7 and Algorithm 8, and compare their performance to that of Algorithm~\ref{alg:fr} and Algorithm~\ref{alg:mr} to show the effect of reusing bid prices. We test Algorithm 7 under two cases: 1) $r_1=2$, $r_2=1$, 2) $r_1=5$, $r_2=1$, and we set capacity $C=0.8k$. 
In both cases, we observe a significant improvement in regret by warm-starting each epoch $u$ with the bid prices from the last epoch.


We test the performance of Algorithm 8 using the same setting as that of Algorithm 7 where 1) $r_1=2$, $r_2=1$, 2) $r_1=5$, $r_2=1$; capacity $C=0.8k$; horizon length $k=10^3, 2\times 10^3, \dots , 10^4$. We set parameter $U=0, 1, 2, 3, 4$ under both cases. By comparing the results in Figure~\ref{fig:mr} and Figure~\ref{fig:mrv}, we observe that the warm-starting technique is effective in improving the algorithm's performance. In addition, we observe that the value of resolving DLPs in the case of Algorithm 8 is not as significant as that in the case of Algorithm~\ref{alg:mr}.


\begin{figure}[hb] 
\center
\includegraphics[height=6cm,width=8cm]{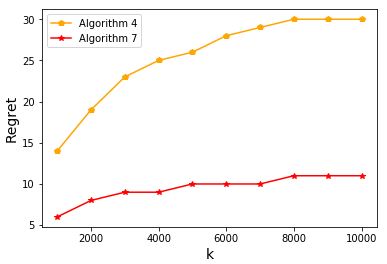}
\includegraphics[height=6cm,width=8cm]{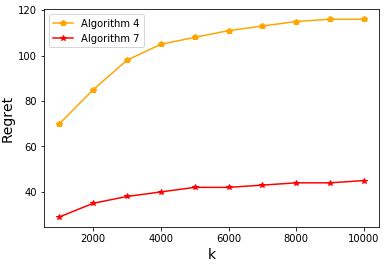}
\\
(a) $C=0.8k$, $r_1=2$ and $r_2=1$ \ \ \ \ \ \ \ \ \ \ \ \ \ \ \ \ \ \ \ (b)    $C=0.8k$, $r_1=5$ and $r_2=1$
\caption{Regret of Algorithm~\ref{alg:fr} and Algorithm 7 for $k=10^3, 2\times 10^3, \dots , 10^4$.}
\label{fig:frv}
\end{figure}

\begin{figure}[ht] 
\center
\includegraphics[height=6cm,width=8cm]{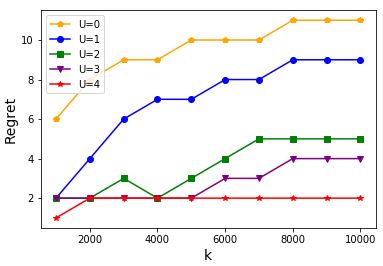}
\includegraphics[height=6cm,width=8cm]{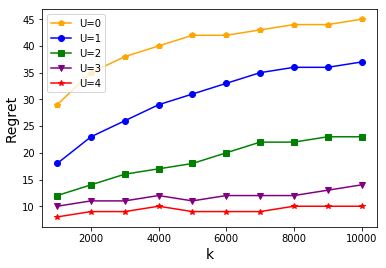}
\\
(a) $C=0.8k$, $r_1=2$ and $r_2=1$ \ \ \ \ \ \ \ \ \ \ \ \ \ \ \ \ \ \ \ (b) $C=0.8k$, $r_1=5$ and $r_2=1$
\caption{Regret of algorithm Algorithm 8 with parameter $U=1,2,3,4$ for $k=10^3, 2\times 10^3, \dots , 10^4$.}
\label{fig:mrv}
\end{figure}


\subsection{Multiple Resources}
Now we consider a NRM problem with multiple resources. Suppose that there are a thousand types of customers and a thousand types of resources. The total number of arrivals of each customer type follows an independent binomial distribution with mean $\frac{k}{1000}$, i.e., the arrival rate is $\lambda_j = \frac{1}{1000}$ for $j\in[1000]$. 
Each type-$i$ customer pays price $r_i$ if accepted. We generate $r_i$ by randomly sampling a value from the set $\{1,2,\dots,10\}$ under a uniform distribution.
We generate the bill-of-material matrix $A\in\real^{1000\times 1000}$ by setting each element
\[a_{ij}=\left\{\begin{matrix}
\; 0, & \text{ with probability } 0.5, \\  
\; 1, & \text{ with probability } 0.5, \\ 
\end{matrix}\right.\]
for all $i \in \{1,2,\dots,1000\}$, $j \in \{1,2,\dots,1000\}$.
We then fix revenue vector $r$ and the bill-of-material matrix $A$ in the experiments. We set $C=0.8k$, and test our Algorithm~\ref{alg:ff}, Algorithm~\ref{alg:fd} and Algorithm~\ref{alg:fr} under horizon length $k=5\times 10^4, 10 \times 10^4, \dots , 5 \times 10^5$.

\begin{figure}[h] 
\center
\includegraphics[height=7cm,width=9cm]{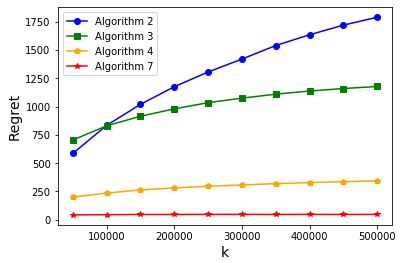}
\caption{Regret of Algorithm~\ref{alg:ff}, Algorithm~\ref{alg:fd}, Algorithm~\ref{alg:fr}, Algorithm 7 for $k=5\times 10^4, 10 \times 10^4, \dots , 5 \times 10^5$.}
\label{fig:multi1}
\end{figure}

Figure \ref{fig:multi1} shows the average regret of Algorithm~\ref{alg:ff}, Algorithm~\ref{alg:fd}, Algorithm~\ref{alg:fr}, Algorithm 7 over $1000$ experiments with the fixed revenue vector $r$ and the fixed bill-of-material matrix $A$. The regret results of different algorithms are aligned with those in the single-resource case: 
1) all algorithms show convergence in regret as $k \rightarrow \infty$; 
2) among Algorithm~\ref{alg:ff}, Algorithm~\ref{alg:fd} and Algorithm~\ref{alg:fr}, Algorithm~\ref{alg:fr} performs the best and Algorithm~\ref{alg:ff} performs the worst; 3) Algorithm 7 improves the performance of Algorithm~\ref{alg:fr} by warm-starting each epoch with the bid prices from the last epoch.

\section{Conclusion}
In this paper, we study the quantity-based network revenue management problem and propose near-optimal algorithms with provable performance guarantees. In particular, our algorithm is the first to achieve $o(\sqrt{T})$ regret bound among all the algorithms that do not solve any deterministic linear programs. Moreover, our algorithms are amenable to implementations in online platforms as the algorithms do not perform matrix inversions, and their memory consumption does not scale with the online traffic.




\bibliographystyle{informs2014}
\bibliography{reference}

\newpage

\begin{APPENDICES}
\section{Proofs for Algorithm~\ref{alg:ff} }\label{ap:ff}

Without loss of generality, we prove lemmas, propositions and theorems with respect to Algorithm~\ref{alg:ff} with  $$t_1=1, \quad t_2= t^*$$ and initial capacity $C$, for any $t^* \in \{1, 2,\ldots, T\}$. Let $C_\text{max} = \max_{i \in [m]}C_i$ and $C_\text{min} = \min_{i \in [m]}C_i$ be the maximum and minimum of the initial capacity values, respectively. Then other notations in Algorithm~\ref{alg:ff} become 
\[L=T-t_1+1 =T,\]
\begin{equation} \label{eq:bar_theta_C}
\bar \theta = \frac{C_{\max}}{C_{\min}} \sum_{i=1}^m \bar{\alpha_i} ,
\end{equation}
\begin{equation} \label{eq:ogd_app_paraD}
D = \bar \theta \sqrt{m} = \frac{C_{\max}}{C_{\min}} \sum_{i=1}^m \bar{\alpha_i} \sqrt{m}, 
\end{equation}
\begin{equation}\label{eq:ogd_app_paraG}
G = \frac{C_{\max}}{T}\sqrt{m} + \bar{a}\sqrt{m},
\end{equation}
and
\[ g_t(\theta) = \sum_{i\in[m]} \theta_i \left(\frac{C_i}{T} - y_{t} a_{i,j(t)} \right).\]


Suppose Algorithm~\ref{alg:ff} stops at the beginning of period $\tau$. In particular, when the algorithm stops due to insufficient remaining resources at the beginning of period $t$, namely due to some $j$ that violates $A_j \leq B(t)$, we have $\tau = t$.
Otherwise, we have $\tau = t^*+1$.





\subsection{Lemma~\ref{lm:OGD'}}

\begin{lemma}\label{lm:OGD'}
There exists some random variable $I \in [m]$ such that
\[ \sum_{t=1}^{\tau-1} g_t(\theta^{(t)}) \leq \alpha \frac{\tau-1}{T} C_I - \alpha\frac{t^*}{T} C_I  + \alpha \bar{a} + \left(\rho_1  \frac{C_\text{max}^2}{C_\text{min} T} + \rho_0 \frac{C_\text{max}}{C_\text{min}} \right) \sqrt{t^*}  \]
almost surely, for any $\alpha \in [0,  \bar{\theta}] $, where $\bar a$ is defined in \eqref{eq:bara}, and  $\rho_0, \rho_1 > 0$ do not depend on $C_\text{max}$, $C_\text{min}$ or $T$.  
\end{lemma}

\begin{proof}{Proof of Lemma~\ref{lm:OGD'}.}

According to the OGD property in Proposition~\ref{prop:ogd}, we have
\begin{equation}\label{eq:OGD}
 \sum_{t=1}^{\tau-1} g_t(\theta^{(t)}) \leq \sum_{t=1}^{\tau-1} g_t(\theta) + \rho \sqrt{\tau-1}
\end{equation}
for all $\theta \in [0, \bar \theta]^m$, where $\rho = \frac{3}{2} G D$. 


Consider the following two cases.

\begin{itemize}
\item Case 1: $\tau-1 < t^*$. Let $\mathbf{e}_i$ denote a unit vector with the $i$-th component being one. For any $\alpha \in [0, \bar \theta]$, since $\alpha \cdot \frac{t^*}{T} \cdot \mathbf{e}_i \in [0, \bar \theta]^m$, we use \eqref{eq:OGD} to obtain
\begin{equation}\label{eq:ogd_proof:1}
\sum_{t=1}^{\tau-1} g_t(\theta^{(t)}) \leq \sum_{t=1}^{\tau-1} g_t(\alpha \cdot \frac{t^*}{T} \cdot \mathbf{e}_i) + \rho \sqrt{\tau-1}    
\end{equation}
for all $i \in [m]$.

Conditioned on $\tau-1 < t^*$, since some resource $I$ has been depleted before the end of the horizon, we have
$\sum_{t=1}^{\tau-1} y_t a_{I, j(t)} \geq C_I - \max_{j\in[n]} a_{I,j} \geq C_I - \bar{a} $, and 
\begin{align}
\sum_{t=1}^{\tau-1} g_t(\alpha \cdot \frac{t^*}{T} \cdot \mathbf{e}_I) & = \sum_{t=1}^{\tau-1} \alpha \frac{t^*}{T} \left( \frac{C_I}{T} - y_t a_{I,j(t)} \right) \nonumber\\
& = \sum_{t=1}^{\tau-1} \alpha  \frac{t^*}{T} \frac{C_I}{T} - \sum_{t=1}^{\tau-1} y_t \alpha  \frac{t^*}{T} a_{I,j(t)}\nonumber \\
& = (\tau-1) \alpha  \frac{t^*}{T} \frac{C_I}{T} -  \alpha  \frac{t^*}{T} \sum_{t=1}^{\tau-1} y_t a_{I,j(t)}\nonumber \\
& \leq (\tau-1) \alpha  \frac{t^*}{T} \frac{C_I}{T} -  \alpha  \frac{t^*}{T} (C_I - \bar{a})\nonumber\\
& = (\tau-1) \alpha  \frac{t^*}{T} \frac{C_I}{T} -  \alpha \frac{t^*}{T} C_I +  \alpha  \frac{t^*}{T} \bar{a}\nonumber\\
& \leq (\tau-1) \alpha  \frac{C_I}{T} -  \alpha  \frac{t^*}{T} C_I + \alpha \bar{a},
\label{eq:ogd_proof:2}
\end{align} 
where the first equality is by definition of $g_t(\cdot)$, the first inequality is because resource $I$ is depleted by time $\tau-1$, and the last inequality follows since $t^*\leq T$. 

Combining \eqref{eq:ogd_proof:1} and \eqref{eq:ogd_proof:2}, we obtain
\[ \sum_{t=1}^{\tau-1} g_t(\theta^{(t)}) 
\leq \alpha \frac{\tau-1}{T} C_I  - \alpha\frac{t^*}{T} C_I + \alpha \bar{a} + \rho \sqrt{t^*}.\]

\item Case 2: $\tau -1 = t^*$. Again, we use \eqref{eq:OGD} to obtain
\begin{align*}
    \sum_{t=1}^{\tau-1} g_t(\theta^{(t)}) & \leq \sum_{t=1}^{\tau-1} g_k(0) + \rho \sqrt{\tau-1} \nonumber\\
    & = 0 + \rho \sqrt{\tau-1}\nonumber\\
    & = \alpha \frac{t^*}{T} C_I  - \alpha\frac{t^*}{T} C_I + \rho \sqrt{t^*} \nonumber\\
    & = \alpha \frac{\tau-1}{T} C_I
    - \alpha\frac{t^*}{T} C_I + \rho \sqrt{t^*} \\
    &\leq \alpha \frac{\tau-1}{T} C_I
    - \alpha\frac{t^*}{T} C_I + \alpha \bar{a} + \rho \sqrt{t^*} . 
\end{align*}

\end{itemize}

Finally, from \eqref{eq:ogd_app_paraD} and \eqref{eq:ogd_app_paraG}, it is easy to derive
\[ 
\rho = \frac{3}{2} G D = \rho_1 \frac{C_\text{max}^2}{C_\text{min} T} + \rho_0 \frac{C_\text{max}}{C_\text{min}},\]
where $\rho_0$ and $\rho_1$ only depend on $m$, the coefficients in $A$, and $r_1,\ldots,r_n$.

This completes the proof. \halmos
\end{proof}

\subsection{Lemma~\ref{lm:martingale}}

Consider the Lagrangian relaxation of \eqref{eq:DLP1}, namely
\begin{align}
\begin{split} \label{eq:dualDLP}
\mathsf{LP}(\theta) = \max_x &\;\; T \sum_{j\in[n]} r_j x_j + T \sum_{i \in [m]} \theta_i( \frac{C_i}{T} - \sum_{j \in [n]} x_j a_{ij}) \\
\text{s.t. } &\;\; 0 \leq x_j \leq \lambda_j. \quad \forall j \in [n]
\end{split}
\end{align}
Define $\{M_t\}_{t=1,\ldots,\tau-1}$ as the stochastic process
\begin{equation}\label{eq:martingale}
    M_t := \sum_{s=1}^t \left[ \frac{\mathsf{LP}(\theta^{(s)})}{T} - y_s r_{j(s)} - g_s(\theta^{(s)})\right].
\end{equation}

\begin{lemma}\label{lm:martingale}
The stochastic process $M_t$ as defined in \eqref{eq:martingale} is a martingale for $t=1,2,\ldots,\tau-1$.
\end{lemma}
\begin{proof}{Proof of Lemma~\ref{lm:martingale}.}
Let $\mathcal{H}_t$ denote all the information in periods $1,\ldots,t$ where $t \leq \tau$.
\begin{align*}
\mathbf{E}[M_t | \mathcal{H}_{t-1}] &= \mathbf{E}\left[\sum_{s=1}^t \left[ \frac{\mathsf{LP}(\theta^{(s)})}{T} - y_s r_{j(s)} - g_s(\theta^{(s)})\right]| \mathcal{H}_{t-1}\right]\\
& = \mathbf{E}\left[\sum_{s=1}^{t-1} \left[ \frac{\mathsf{LP}(\theta^{(s)})}{T} - y_s r_{j(s)} - g_s(\theta^{(s)})\right]| \mathcal{H}_{t-1}\right] +\mathbf{E}\left[ \frac{\mathsf{LP}(\theta^{(t)})}{T} - y_t r_{j(t)} - g_t(\theta^{(t)})| \mathcal{H}_{t-1}\right].
\end{align*}

For the first term, since $\sum_{s=1}^{t-1} \left[ \mathsf{LP}(\theta^{(s)})/T - y_s r_{j(s)} - g_s(\theta^{(s)}) \right] \in \mathcal{H}_{t-1}$, we have
\begin{equation} \label{eq:0}
    \mathbf{E}\left[\sum_{s=1}^{t-1} \left[ \frac{\mathsf{LP}(\theta^{(s)})}{T} - y_s r_{j(s)} - g_s(\theta^{(s)})\right]| \mathcal{H}_{t-1}\right]=\sum_{s=1}^{t-1} \left[ \frac{\mathsf{LP}(\theta^{(s)})}{T} - y_s r_{j(s)} - g_s(\theta^{(s)})\right]=M_{t-1}.
\end{equation}

For the second term, let ${x}^*_j(\theta)$ for $j\in[n]$ denote an optimal solution to \eqref{eq:dualDLP}. The structure of \eqref{eq:dualDLP} is simple enough that we can calculate for all $j \in [n]$:
\begin{itemize}
    \item given $r_j > \sum_{i\in[m]} \theta_i^{(t)} a_{ij}$, we have $x^*_j(\theta^{(t)}) = \lambda_j$;
    \item given $r_j \leq \sum_{i\in[m]} \theta_i^{(t)} a_{ij}$, we have $x^*_j(\theta^{(t)}) = 0$.
\end{itemize}
By definition, we know 
$y_t = \mathbf{1}(r_{j(t)} > \sum_{i\in[m]} \theta_i^{(t)} a_{i,j(t)})$. Therefore, we have $y_t = x^*_{j(t)}(\theta^{(t)}) /\lambda_{j(t)}$. 

Notice that $\theta^{(t)} \in \mathcal{H}_{t-1}$, since the value of $\theta^{(t)}$ has been determined at the end of period $t-1$. Thus, by taking expectation over all customer types $j\in [n]$ for $j(t)$, we have
\begin{align}
&\mathbf{E}[y_tr_{j(t)}| \mathcal{H}_{t-1}] \nonumber\\
= &\mathbf{E}[ \frac{x^*_{j(t)}(\theta^{(t)}) }{\lambda_{j(t)}} r_{j(t)}| \mathcal{H}_{t-1}] \nonumber\\
= &\sum_{j \in [n]} \mathbf{P}(j(t) = j) \frac{x^*_{j}(\theta^{(t)}) }{\lambda_{j}} r_{j} \nonumber\\
= &\sum_{j \in [n]} \lambda_j \frac{x^*_{j}(\theta^{(t)}) }{\lambda_{j}} r_{j} \nonumber\\
= &\sum_{j \in [n]} r_j x^*_j(\theta^{(t)}).   \label{eq:1}
\end{align}

\begin{align} 
&\mathbf{E}[g_t(\theta^{(t)})| \mathcal{H}_{t-1}] \nonumber\\
=&\mathbf{E}[\sum_{i\in[m]} \theta_i^{(t)} \left(\frac{C_i}{T} - y_{t} a_{i,j(t)} \right)| \mathcal{H}_{t-1}] \nonumber\\
=&\sum_{i\in[m]} \theta_i^{(t)} \frac{C_i}{T} - \sum_{i\in[m]} \theta_i^{(t)} \mathbf{E}[  y_{t} a_{i,j(t)} | \mathcal{H}_{t-1}] \nonumber\\
=&\sum_{i\in[m]} \theta_i^{(t)} \frac{C_i}{T} - \sum_{i\in[m]} \theta_i^{(t)} \mathbf{E}[ \frac{x^*_{j(t)}(\theta^{(t)})}{\lambda_{j(t)}} a_{i,j(t)} | \mathcal{H}_{t-1}] \nonumber\\
=&\sum_{i\in[m]} \theta_i^{(t)} \frac{C_i}{T} - \sum_{i\in[m]} \theta_i^{(t)} \sum_{j \in [n]} \mathbf{P}(j(t) = j) \frac{x^*_{j}(\theta^{(t)})}{\lambda_{j}} a_{i,j} \nonumber\\
=&\sum_{i\in[m]} \theta_i^{(t)} \frac{C_i}{T} - \sum_{i\in[m]} \theta_i^{(t)} \sum_{j \in [n]} \lambda_j \frac{x^*_{j}(\theta^{(t)})}{\lambda_{j}} a_{i,j} \nonumber\\
=&\sum_{i \in [m]} \theta_i^{(t)}\frac{C_i}{T}-\sum_{i\in [m]} \sum_{j\in [n]} \theta_i^{(t)} x^*_j(\theta^{(t)}) a_{ij}. \label{eq:2}
\end{align}

Combining \eqref{eq:1} and \eqref{eq:2}, we have
\begin{equation}
   \mathbf{E}[y_t r_{j(t)} +g_t(\theta^{(t)})| \mathcal{H}_{t-1}] = \sum_{j \in [n]} r_j x^*_j(\theta^{(t)}) + \sum_{i \in [m]} \theta_i^{(t)}\frac{C_i}{T}-\sum_{i\in [m]} \sum_{j\in [n]} \theta_i^{(t)} x^*_j(\theta^{(t)}) a_{ij} =\frac{\mathsf{LP}(\theta^{(t)})}{T}.
\end{equation}

Hence, we have
\begin{equation} \label{eq:3}
    \mathbf{E}\left[ \left[ \frac{\mathsf{LP}(\theta^{(t)})}{T} - y_t r_{j(t)} - g_t(\theta^{(t)})\right] \big| \mathcal{H}_{t-1}\right]=0
\end{equation}

Combining \eqref{eq:0} and \eqref{eq:3}, we have
\begin{equation}
    \mathbf{E}[M_t | \mathcal{H}_{t-1}] = M_{t-1},
\end{equation}
which shows that $M_t$ is a martingale.
\halmos
\end{proof}

\subsection{Lemma~\ref{lm:az}}
\begin{lemma}[Azuma's Inequality]\label{lm:az}
Suppose $M_t$ is a martingale, and 
\begin{equation}
    \vert M_t-M_{t-1} \vert \leq c_t
\end{equation}
almost surely, then for any positive integer $N$ and any positive $a \in \mathbb{R}$, we have
\begin{equation}
     \mathbf{P}(M_t \leq a) \geq 1- \exp\left(- \frac{a^2}{2 \sum_{k=1}^t c_k^2}\right).
\end{equation}

\end{lemma}

\subsection{Lemma~\ref{lm:Mt}}
\begin{lemma}\label{lm:Mt}
Consider $M_t$ as defined in \eqref{eq:martingale}. With probability at least $1-\frac{1}{K}$, we have
\[ M_t \leq \left(\rho_2 \frac{C_\text{max}}{C_\text{min}} + \rho_3\right) \sqrt{\tau-1} (\log K)^{\frac{1}{2}} \]
for all $t=1,2,\ldots,\tau-1$, where $\rho_2$ and $\rho_3$ do not depend on $C_\text{max}$, $C_\text{min}$ or $K$. \end{lemma}
\begin{proof}{Proof of Lemma~\ref{lm:Mt}.}
We apply Azuma's inequality to obtain
\begin{equation}\label{eq:M_t:0}
    \mathbf{P}\left(M_t \leq \left(\rho_2 \frac{C_\text{max}}{C_\text{min}} + \rho_3\right) \sqrt{\tau-1}(\log K)^{\frac{1}{2}}\right) \geq 1-\exp\left(- \frac{\left(\rho_2 \frac{C_\text{max}}{C_\text{min}} + \rho_3\right)^2 (\tau-1) \log K}{2 \sum_{k=1}^t c^2}\right),
\end{equation}
where $c$ satisfies $\vert M_k-M_{k-1} \vert \leq c_k$. We next find the proper values of $\rho_2$, $\rho_3$ and $c$.

Let ${x}^*_j(\theta)$ for $j\in[n]$ denote an optimal solution to $\LP(\theta)$ in \eqref{eq:dualDLP}. 
We have $0\leq {x}^*_j(\theta) \leq \lambda_j$. 

By definition of $M_t$, we know
\begin{align*}
M_t-& M_{t-1} = \frac{\LP(\theta^{(t)})}{T} - y_t r_{j(t)} - g_t(\theta^{(t)}) \\
&=\frac{1}{T} \left( T \sum_{j\in[n]} r_j x_j^*(\theta^{(t)}) + T \sum_{i \in [m]} \theta_i^{(t)}( \frac{C_i}{T}- \sum_{j \in [n]} a_{ij} {x}^*_j(\theta^{(t)}) ) \right) - y_t r_{j(t)} - \sum_{i\in[m]} \theta_i^{(t)} \left(\frac{C_i}{T} - y_{t} a_{i,j(t)} \right) \\
&= \sum_{j\in[n]} r_j {x}^*_j(\theta^{(t)}) - \sum_{i \in [m]} \theta_i^{(t)} \sum_{j \in [n]} a_{ij} {x}^*_j(\theta^{(t)})  - y_t r_{j(t)} + y_t \sum_{i \in [m]} a_{i, j(t)} \theta_i^{(t)}.
\end{align*}

Recall $\bar{a}=\max_{i\in[m],j\in[n]}a_{ij}$ and $\theta^{(t)} \in [0,\bar{ \theta}]^m$. 
We remove the non-positive terms in the above equation to obtain
\begin{align}\label{eq:M_t:1}
M_t-M_{t-1} 
&\leq \sum_{j\in[n]} r_j {x}^*_j(\theta^{(t)}) +  y_t \sum_{i \in [m]} a_{i, j(t)} \theta_i^{(t)} \nonumber \\
&\leq \Vert r \Vert_{\infty} \sum_{j\in[n]}  {x}^*_j(\theta^{(t)}) + m \bar{a} \bar \theta  \nonumber \\
&\leq \Vert r \Vert_{\infty} \sum_{j\in[n]}  \lambda_j + m \bar{a} \bar \theta  \nonumber \\
&= \Vert r \Vert_{\infty} + m \bar{a} \bar \theta .
\end{align}

On the other hand, we remove the non-negative terms to obtain
\begin{align}
    M_t-M_{t-1} & \geq - \sum_{i \in [m]} \theta_i^{(t)} \sum_{j \in [n]} a_{ij} {x}^*_j(\theta^{(t)})  - y_t r_{j(t)} \nonumber \\
    & \geq - m \bar{\theta} \bar{a} \sum_{j \in [n]} {x}^*_j(\theta^{(t)})  - \Vert r \Vert_\infty \nonumber \\    
    & \geq - m \bar{\theta} \bar{a} \sum_{j \in [n]} \lambda_j  - \Vert r \Vert_\infty \nonumber \\
    & \geq - m \bar{\theta} \bar{a} - \Vert r \Vert_\infty. \label{eq:M_t:3}
\end{align}

In sum, we have
\[ |M_t - M_{t-1}| \leq \Vert r \Vert_\infty + m \bar{a} \bar{\theta}.\]

We set
\[c = \Vert r \Vert_\infty + m \bar{a} \bar{\theta},\]
\[ \rho_2 = \sqrt{2} m \bar{a} \sum_{i=1}^m \bar{\alpha}_i,\]
\[ \rho_3 = \sqrt{2} \Vert r\Vert_\infty\]
in Equation \eqref{eq:M_t:0}. We can then simplify \eqref{eq:M_t:0} to 


\begin{align}\nonumber
    \mathbf{P}\left(M_t \leq \left(\rho_2 \frac{C_\text{max}}{C_\text{min}} + \rho_3\right) \sqrt{\tau-1}(\log K)^{\frac{1}{2}}\right) &\geq 1-\exp\left(- \frac{\left(\rho_2 \frac{C_\text{max}}{C_\text{min}} + \rho_3\right)^2 (\tau-1) \log K}{2 \sum_{k=1}^t (\Vert r \Vert_\infty + m \bar{a} \bar{\theta})^2}\right)
    \\
    &\geq 1-\exp\left(- \frac{\left( \sqrt{2} m \bar{a} \frac{C_{\max}}{C_{\min}} \sum_{i=1}^m \bar{\alpha_i} + \sqrt{2} \Vert r \Vert_\infty\right)^2 (\tau-1) \log K}{2 \sum_{k=1}^t (\Vert r \Vert_\infty + m \bar{a} \frac{C_{\max}}{C_{\min}} \sum_{i=1}^m \bar{\alpha_i})^2}\right) \nonumber\\
    &=1-\exp\left(- \frac{\tau-1}{t} \log K\right)
    \geq 1-\frac{1}{K}
\end{align}
for all $t=1,2,\ldots,\tau-1$.

This completes the proof. \halmos
\end{proof}

\subsection{Lemma~\ref{lm:alpha}}

\begin{lemma}\label{lm:alpha}
Given the definition of $\bar{\alpha}$ in \eqref{eq:bar_theta_C}, we have
\[ \frac{V^{\DLP}}{C_{\min}} \leq \bar{\theta}. \]
\end{lemma}


\begin{proof}{Proof of Lemma~\ref{lm:alpha}.}
Consider $\LP(\theta)$ as defined in \eqref{eq:dualDLP}. By weak duality, we have
$\LP(\theta) \geq V^{\DLP}$
for any $\theta \geq 0$.
Recall that $\bar{\alpha_i} = \max_{j: a_{ij} \neq 0} r_j/a_{ij}$ and $\bar{\theta}
= \sum_{i\in[m]} \bar{\alpha_i} C_{\max}/C_{\min} .$
By definition, we know
\[\bar{\alpha_i}  a_{ij} \geq r_j,\; \forall j \text{ and } a_{ij} \neq 0 . \]

Now let $\theta_i = \bar{\alpha_i}$ for $i\in[m]$ and let $x^*(\theta)$ denote an optimal solution to \eqref{eq:dualDLP}. We can derive
\begin{align*}
    V^{\DLP} \leq \LP(\theta) 
    &= T \sum_{j\in[n]} (r_j - \sum_{i\in[m]} \bar{\alpha_i} a_{ij})  x_j^*(\theta)
    + \sum_{i\in[m]} \bar{\alpha_i} C_i
    \leq \sum_{i\in[m]} \bar{\alpha_i} C_{\max} .
\end{align*}
The last inequality follows since 1) the first term on the left hand side is non-positive, and 2) $C_i \leq C_{\max}$ for all $i \in [m]$ by definition.
Divide both sides by $C_{\min}$, and we obtain the result. \halmos
\end{proof}

\subsection{Proposition~\ref{prop:ff:regret}} 

Now we prove the following proposition, which is more general than Proposition~\ref{prop:ff:regret}.
Let $V^{\FF}(t^*)$ denote the revenue of  Algorithm~\ref{alg:ff} with start time $t_1=1$, end time $t_2= t^*$ and initial capacity $C$.

\begin{proposition} \label{prop:ff+:regret}
With probability at least $1-\frac{1}{K}$ where $K = \Omega(T)$, we have
\begin{equation}
    \frac{t^*}{T} V^{\DLP} - V^{\FF}(t^*) \leq \left(\rho_2 \frac{C_\text{max}}{C_\text{min}} + \rho_3\right) \sqrt{t^*}(\log K)^\frac{1}{2} + \left(\rho_1  \frac{C_\text{max}^2}{C_\text{min} T} + \rho_0 \frac{C_\text{max}}{C_\text{min}} \right) \sqrt{t^*} + \rho_4 \frac{C_\text{max}}{C_\text{min}},
\end{equation}
where $\rho_0, \rho_1, \ldots, \rho_4$ do not depend on $C_\text{max}$, $C_\text{min}$, $T$ or $K$. 
\end{proposition}

\begin{proof}{Proof of  Proposition~\ref{prop:ff+:regret}.}

By Proposition~\ref{lm:Mt} and definition of $M_t$, we have with probability at least $1-\frac{1}{K}$ that
\begin{equation}\label{eq:prop3:1}
    \sum_{s=1}^t \left[ \frac{\mathsf{LP}(\theta^{(s)})}{T} - y_s r_{j(s)} - g_s(\theta^{(s)})\right]
    \leq \left(\rho_2 \frac{C_\text{max}}{C_\text{min}} + \rho_3\right) \sqrt{\tau-1}(\log K)^{\frac{1}{2}}.
\end{equation}
By Lemma \ref{lm:OGD'}, we know there exists some random variable $I \in [m]$ such that
\begin{equation}\label{eq:prop3:2}
\sum_{t=1}^{\tau-1} g_t(\theta^{(t)}) \leq \alpha \frac{\tau-1}{T} C_I - \alpha\frac{t^*}{T} C_I  + \alpha \bar{a} + \left(\rho_1  \frac{C_\text{max}^2}{C_\text{min} T} + \rho_0 \frac{C_\text{max}}{C_\text{min}} \right) \sqrt{t^*}
\end{equation}
almost surely, for any $\alpha \in [0,  \bar{\theta}] $. 

Recall that $\mathsf{LP}(\theta)$ is the Lagrangian relaxation defined in Equation \eqref{eq:dualDLP}.
By weak duality, we have $\mathsf{LP}(\theta) \geq V^{\text{DLP}}$ for all $\theta \geq 0$.
Combining this with \eqref{eq:prop3:1} and \eqref{eq:prop3:2}, we can derive that with probability at least $1-\frac{1}{K}$ that 
\begin{align*}
    \frac{\tau-1}{T} V^{\text{DLP}} &\leq  \sum_{t=1}^{\tau-1} \frac{\mathsf{LP}(\theta^{(t)})}{T} \\
    &\leq \sum_{t=1}^{\tau-1} y_t r_{j(t)} + \sum_{t=1}^{\tau-1} g_t(\theta^{(t)})  + \left(\rho_2 \frac{C_\text{max}}{C_\text{min}} + \rho_3\right) \sqrt{t^*} (\log K)^{\frac{1}{2}} \\
    &\leq \sum_{t=1}^{\tau-1} y_t r_{j(t)} + \alpha \frac{\tau-1}{T} C_I - \alpha \frac{t^*}{T} C_I + \alpha \bar{a}  + \left(\rho_1  \frac{C_\text{max}^2}{C_\text{min} T} + \rho_0 \frac{C_\text{max}}{C_\text{min}} \right) \sqrt{t^*} + \left(\rho_2 \frac{C_\text{max}}{C_\text{min}} + \rho_3\right) \sqrt{t^*} (\log K)^{\frac{1}{2}}
\end{align*}
for any $\alpha \in [0, \bar{\theta}] $.
Now we set $\alpha = \bar{\theta}$. By Lemma~\ref{lm:alpha}, we know $ \alpha \cdot C_I = \bar{\theta} \cdot C_I \geq V^{\DLP}$.
We continue to derive
\begin{align*}
& \frac{\tau-1}{T} V^{\text{DLP}} \\
\leq & \sum_{t=1}^{\tau-1} y_t r_{j(t)} + \bar{\theta} \frac{\tau-1}{T} C_I - \bar{\theta} \frac{t^*}{T} C_I + \bar{\theta} \bar{a}  + \left(\rho_1  \frac{C_\text{max}^2}{C_\text{min} T} + \rho_0 \frac{C_\text{max}}{C_\text{min}} \right) \sqrt{t^*} + \left(\rho_2 \frac{C_\text{max}}{C_\text{min}} + \rho_3\right) \sqrt{t^*} (\log K)^{\frac{1}{2}}\\
\leq & \sum_{t=1}^{\tau-1} y_t r_{j(t)} - \bar{\theta}  C_I \left( \frac{t^*}{T} - \frac{\tau-1}{T} \right) + \bar{\theta} \bar{a}  + \left(\rho_1  \frac{C_\text{max}^2}{C_\text{min} T} + \rho_0 \frac{C_\text{max}}{C_\text{min}} \right) \sqrt{t^*} + \left(\rho_2 \frac{C_\text{max}}{C_\text{min}} + \rho_3\right) \sqrt{t^*} (\log K)^{\frac{1}{2}}\\
\leq & \sum_{t=1}^{\tau-1} y_t r_{j(t)} - V^\DLP \left( \frac{t^*}{T} - \frac{\tau-1}{T} \right) + \bar{\theta} \bar{a}  + \left(\rho_1  \frac{C_\text{max}^2}{C_\text{min} T} + \rho_0 \frac{C_\text{max}}{C_\text{min}} \right) \sqrt{t^*} + \left(\rho_2 \frac{C_\text{max}}{C_\text{min}} + \rho_3\right) \sqrt{t^*} (\log K)^{\frac{1}{2}}\\
\leq & V^\FF(t^*) - V^\DLP \left( \frac{t^*}{T} - \frac{\tau-1}{T} \right) + \bar{\theta} \bar{a}  + \left(\rho_1  \frac{C_\text{max}^2}{C_\text{min} T} + \rho_0 \frac{C_\text{max}}{C_\text{min}} \right) \sqrt{t^*} + \left(\rho_2 \frac{C_\text{max}}{C_\text{min}} + \rho_3\right) \sqrt{t^*} (\log K)^{\frac{1}{2}}.
\end{align*}

\begin{align*}
\Longrightarrow  \frac{t^*}{T} V^\DLP - V^\FF(t^*) \leq  \bar{\theta} \bar{a}  + \left(\rho_1  \frac{C_\text{max}^2}{C_\text{min} T} + \rho_0 \frac{C_\text{max}}{C_\text{min}} \right) \sqrt{t^*} + \left(\rho_2 \frac{C_\text{max}}{C_\text{min}} + \rho_3\right) \sqrt{t^*} (\log K)^{\frac{1}{2}}.
\end{align*}

Finally, we set $\rho_4 = \sum_{i=1}^m \bar{\alpha_i} \bar{a}$ to obtain
\[ \bar{\theta}\bar{a}  = \frac{C_{\max}}{C_{\min}} \sum_{i=1}^m \bar{\alpha_i} \bar{a} = \frac{C_{\max}}{C_{\min}} \rho_4, \]
which completes the proof.

\halmos

\end{proof}

Proposition~\ref{prop:ff:regret} follows from 
Proposition~\ref{prop:ff+:regret}, where $t^*=T$, $K =T$ and $V^{\FF}(T) = V^{\FF}$.


\subsection{Lemma~\ref{lm:asmp:1}}

Recall the definition of $\LP(b,v)$ in \eqref{eq:lp_b_u}:
\begin{align*}
\begin{split}
  \LP(b,v) := \max_{z,\varepsilon} \;\;& \sum_{j=1}^n r_j z_j  \\
  \text{s.t. } & \sum_{j=1}^n A_j z_j + \varepsilon_i = b \\
  & 0 \leq z_j  \leq v_j,\;\; \forall j \in [n]\\
  & \varepsilon_i \geq 0, \;\; \forall i \in [m],
\end{split}
\end{align*}
and matrix $\bar{A} = [A_1, \ldots, A_n, \mathbf{e}_1, \ldots, \mathbf{e}_m] $, vector $\bar{r} = [r_1,\ldots,r_n, 0, \ldots, 0]^\top \in \real^{m+n}$.


\begin{proof}{Proof of Lemma~\ref{lm:asmp:1}.}
Let $\tilde{w}^*= (z^*_1, \ldots, z^*_n, \varepsilon_1^*, \ldots, \varepsilon_m^*)$ be an optimal solution to $\LP(b,v)$ where $\varepsilon^*_i = b - \sum_{j=1}^n A_j z^*_j$ for $i\in[m]$. 

First, suppose $\tilde{w}^*$ is a basic feasible solution (BFS). Let $B \subset [m+n]$ be the set of indices for the basic variables, and $N := [m+n] - B$ the set of indexes for the non-basic variables. By definition, we know $|B|=m$ given there exists $m$ constraints in \eqref{eq:lp_b_u}. In addition, the \emph{reduced cost} of the basic variables is zero, i.e., we have
\[
\bar{r}_l - \bar{r}_B^\top \bar{A}_B^{-1} \bar{A}_l = 0
\text{ for } l\in B.
\]
\end{proof}

Consider moving from $\tilde{w}^*$ along the $k$-th \emph{basic direction} $d_k$ where $k\in N$. By definition, $\bar{A} (\tilde{w}^* + d) = b$, which implies $\bar{A} d = \bar{A}_B d_B + \bar{A}_N d_N = 0$. In particular, we have two types of basic feasible directions.

\begin{itemize}

\item $d_N = \mathbf{e}_k$ where $k$ corresponds to the index of non-basic variable $z^*_j = 0$ or $\varepsilon^*_i = 0$. In this case, $d_B = \bar{A}_B^{-1} \bar{A}_k$, and moving one unit along $d_k$ results in a change of revenue $\bar{r}_k - \bar{r}_B^\top \bar{A}_B^{-1}\bar{A}_k$;

\item $d_N = -\mathbf{e}_k$ where $k$ corresponds to the index of non-basic variable $z^*_j = u_j$. In this case, $d_B = \bar{A}_B^{-1} \bar{A}_k$, and moving one unit along $d_k$ results in a change of revenue $-(\bar{r}_k - \bar{r}_B^\top \bar{A}_B^{-1}\bar{A}_k )$.

\end{itemize}

Define bid price vector $p := \bar{r}_B^\top \bar{A}_B^{-1} \in \real^m$. Under Assumption~\ref{asmp:1}, we know there are at most $m$ columns such that $\bar{r}_k - p \bar{A}_k = 0$. 
Since we already have $\bar{r}_l - p \bar{A}_l = 0$ for $l\in B$ and $|B| =m$, we know $\bar{r}_k - p \bar{A}_k \neq 0$ for all $k\in N$. In addition, since $\tilde{w}^*$ is the optimal solution, we know all basic directions $d_k$ are strict gradient descent directions, namely, 
\[
\bar{r}^\top d_k < 0.
\]
In particular, we have $\bar{r}_k - \bar{r}_B^\top \bar{A}_B^{-1}\bar{A}_k < 0$ for $d_N = \mathbf{e}_k$, and $\bar{r}_B^\top \bar{A}_B^{-1}\bar{A}_k - \bar{r}_k < 0$ for $d_N = -\mathbf{e}_k$. Thus, we know $\tilde{w}^*$ is unique.

Second, suppose $\tilde{w}^*$ is not a BFS. Then following the same argument, we know there exists a BFS optimal solution that is unique, which results in a violation.


Given any selection of basis index $B$ and $p=\bar{r}_B^\top \bar{A}_B^{-1}$, define $\delta(p) := - \min_{k \in N} |\bar{r}_k - p\bar{A}_k| $, i.e., the maximum value of $\bar{r}^\top d_k$ for $k\in N$. 
Since there are finitely many selections of basis, we know there exists a constant $\delta_{\min} < 0$ such that $\delta(p) \leq \delta_{\min}$ for any price vector $p= \bar{r}_B^\top \bar{A}_B^{-1}$.


Now consider any feasible solution $\tilde{w}$ of \eqref{eq:lp_b_u}, and direction $d=\frac{ \tilde{w} - \tilde{w}^*}{\Vert  \tilde{w} - \tilde{w}^* \Vert}$. 
Given the basic direction $d_k \in \real^{m+n}$ for $k\in N$ associated with optimal solution $\tilde{w}^*$, we know $d$ can be written as a convex combination of $d_k$, namely, there exist constants $c_k \geq 0$ for $k \in N$ such that 
\[ d = \sum_{k \in N} c_k d_k .\]

We have
\begin{equation*}
    \bar{r}^\top d =  \bar{r}^\top \sum_{k \in N} c_k d_k =  \sum_{k \in N} c_k (\bar{r}^\top d_k) \leq \sum_{k \in N} c_k \delta_{\min},
\end{equation*}
which is guaranteed to be smaller than a fixed constant.

\halmos


\subsection{Proposition~\ref{prop:ff:x}} 
Now we prove a lemma and a proposition that generalize the results of Proposition~\ref{prop:ff:x}.


Let $x^{\FF}_j$ denote the number of type-$j$ customers accepted by Algorithm~\ref{alg:ff} with start time $t_1=1$, end time $t_2= t^*$ and initial capacity $C$.

Let $z^*$ denote an optimal solution to
\[
\max_{z} \left\{\sum_{j=1}^n r_j z_j \mid \sum_{j=1}^n A_j z_j \leq C, \ 0 \leq z_j \leq \lambda_j t^*, \  \forall j \in [n] \right\}
\]

Let $S=\cap_{j=1}^n \{z  \mid 0 \leq z_j \leq \lambda_j t^* \}$, and $Q=\{z \mid \sum_{j=1}^n A_j z_j \leq \frac{t^*}{T} C\}$.

\begin{lemma}\label{lm:ff+:x}
Under Assumption~\ref{asmp:1}, 
if $x \in S \cap Q$ and satisfies
\[
\sum_{j=1}^n  r_j z^*_j - \sum_{j=1}^n  r_j x_j = O( \sqrt{T}(\log K)^{\frac{1}{2}})
\] 
where $K=\Omega(T)$,
then we have
$\Vert x - z^*\Vert =O( \sqrt{T}(\log K)^{\frac{1}{2}})$.
\end{lemma}

\begin{proof}{Proof of Lemma~\ref{lm:ff+:x}.}

Consider $\LP(b,v)$ with capacity vector $b=\frac{t^*}{T} C$ and upper bound vector $v$ where $v_j=\lambda_j t^*$ for $j\in[n]$.
Define $\varepsilon^*_i = b - \sum_{j=1}^n A_j z^*_j$ for $i\in[m]$, and $\tilde{w}^*= (z^*_1, \ldots, z^*_n, \varepsilon_1^*, \ldots, \varepsilon_m^*)$.
By definition, we know $\tilde{w}^*$ is an optimal solution to $\LP(\frac{t^*}{T} C,\lambda t^*)$. 
By Lemma~\ref{lm:asmp:1}, we know $\tilde{w}^*$ is unique, and is a BFS of the LP.


If $x \in S \cap Q$, we define $\tilde{w}=(x_1, \ldots, x_n, \varepsilon_1, \ldots, \varepsilon_m)$, where $ \varepsilon_i=\frac{t^*}{T} C_i - \sum_{j=1}^n A_j x_j$ for $i\in[m]$. We know $\tilde{w}$ is a feasible solution to $\LP(\frac{t^*}{T} C,\lambda t^*)$.
Now define direction $d = \frac{\tilde{w}- \tilde{w}^*}{\Vert \tilde{w}-\tilde{w}^* \Vert}$. By Lemma~\ref{lm:asmp:1}, we know 
$\bar{r}^\top d$ is guaranteed to be smaller than a fixed constant, denoted as $\Delta_{\min}$.



Since $\bar{r}^\top \tilde{w}^* - \bar{r}^\top \tilde{w} = - \bar{r}^\top d \ \Vert \tilde{w}^* - \tilde{w} \Vert $, we know 
\[ 
 \Vert \tilde{w}^* - \tilde{w} \Vert  = - \frac{\bar{r}^\top \tilde{w}^* - \bar{r}^\top \tilde{w}}{\bar{r}^\top d }
\leq \Delta_{\min} (\bar{r}^\top \tilde{w}^* - \bar{r}^\top \tilde{w}).
\]

Since $\sum_{j=1}^n  r_j z^*_j - \sum_{j=1}^n  r_j x_j = O( \sqrt{T}(\log K)^{\frac{1}{2}})$, we know there exist constants $a$ and $T_0$ such that for all $T\geq T_0$, we have 
$\sum_{j=1}^n  r_j z^*_j - \sum_{j=1}^n  r_j x_j \leq a \sqrt{T}(\log K)^{\frac{1}{2}}$. This implies 
$\bar{r}^\top \tilde{w}^* - \bar{r}^\top \tilde{w} \leq a \sqrt{T}(\log K)^{\frac{1}{2}}$, and hence
\[
 \Vert \tilde{w}^* - \tilde{w} \Vert \leq \Delta_{\min}a \sqrt{T}(\log K)^{\frac{1}{2}}.
\]

By definition, we know $\Vert x - z^* \Vert \leq \Vert \tilde{w} - \tilde{w}^* \Vert$. Therefore,
\[
\Vert x - z^* \Vert \leq \Delta_{\min}a \sqrt{T}(\log K)^{\frac{1}{2}} = O(\sqrt{T}(\log K)^{\frac{1}{2}}).
\]

\end{proof}


\begin{proposition} \label{prop:ff+:x}
Under Assumption~\ref{asmp:1}, with probability at least $1-O(\frac{1}{K})$ where $K=\Omega(T)$, we have
\begin{equation}\label{eq:ff+:x}
    |x^{\FF}_j - z_j^*| \leq \rho \sqrt{T}(\log K)^\frac{1}{2} 
\end{equation}
for all $j\in[n]$, where $\rho$ does not depend on $t^*$, $T$ or $K$.
\end{proposition}

\begin{proof}{Proof of Proposition~\ref{prop:ff+:x}.}




Given $K = \Omega(T)$, define $\mu(t)=\sqrt{\frac{t}{2}} (\log K)^\frac{1}{2}$. 
Let $Y_j$ denote the event $\big\{\Lambda_j(t^*) \leq \lambda_j t^* +\mu(t^*) \big\} $. 
Since $\Lambda_j(t^*)$ is a binomial random variable with mean $\lambda_j t^*$.
By Hoeffding's inequality, we have
\[
P(Y_j) = P\big(\Lambda_j(t^*) \leq \lambda_j t^* +\mu(t^*) \big) \geq 1- \exp \bigg(- \frac{2\mu^2(t^*)}{t^*} \bigg) 
\geq 1 - \frac{1}{K}.
\]
Define $S^{'}  = \cap_{j=1}^n \big\{x \mid 0 \leq x_j \leq \lambda_j t^*+ \mu_j(t^*) \big\}$.
Let $Z$ denote the event $\big\{\frac{t^*}{T} V^{\DLP}- V^{\FF} \leq \rho_2 \sqrt{t^*} (\log K)^{\frac{1}{2}} + \rho_1 \sqrt{t^*} \big\}$.
By Proposition \ref{prop:ff+:regret}, we know that $P(Z) \geq 1- 1/K$. 
Therefore, the joint event $(\cap_{j\in[n]} Y_j )\cap Z$ happens with probability at least
\[
P(\cap_{j\in[n]} Y_j \cap Z) = 1 - P(\cup_{j\in[n]} Y_j^\complement \cup Z^\complement) 
\geq 1- \sum_{j\in[n]} P(Y_j^\complement) - P(Z^\complement) = 1 - O(\frac{1}{K}),
\]
where the inequality follows by the union bound.

Consequently, it suffices to prove Proposition \ref{prop:ff+:x} conditional on $(\cap_{j\in[n]} Y_j )\cap Z$. We consider two cases: 
\[
\text{1) } x^{\FF} \in S \cap Q ,
\text{ and 2) } x^{\FF} \in (S'-S) \cap Q ,
\]
where we define $S' =\cap_{j=1}^n \{x \mid 0 \leq x_j \leq \lambda_j T + \sqrt{T}(\log K)^\frac{1}{2} \}$. 


In case 1) where $x^{\FF} \in S \cap Q$, we have $\sum_{j=1}^n  r_j z^*_j - \sum_{j=1}^n  r_j x^{\FF}_j = O(\sqrt{T}(\log K)^{\frac{1}{2}})$
conditioned on event $Z$.
By Lemma~\ref{lm:ff+:x}, we know immediately $\Vert x^{\FF} - z^*\Vert =  O(\sqrt{T}(\log K)^{\frac{1}{2}})$.

In case 2) where $x^{\FF} \in (S'-S) \cap Q$, we first project $x^{\FF}$ onto $S \cap Q$.
Let $x'$ be the projection of $x^{\FF}$ onto $S \cap Q$. Since $x^{\FF} - \sqrt{T}(\log K)^{\frac{1}{2}} \mathbf{e} \in S \cap Q$, where $\mathbf{e}$ denotes a vector of ones in $\real^n$ , by the definition of projection, we know
\begin{equation}\label{eq:ff+:x0}
\Vert x^{\FF} - x' \Vert \leq 
\Vert x^{\FF} - (x^{\FF} - \sqrt{T}(\log K)^{\frac{1}{2}} \mathbf{e}) \Vert 
= \sqrt{n} \sqrt{T}(\log K)^{\frac{1}{2}} ,
\end{equation}
which then implies
\begin{align}
\sum_{j=1}^n  r_j x^{\FF}_j - \sum_{j=1}^n  r_j x'_j
&= \sum_{j=1}^n r_j (x^{\FF}_j - x'_j)  \nonumber \\
&\leq 
\Vert r^\top (x^{\FF} - x') \Vert_1 \nonumber \\
&\leq 
\Vert r \Vert
\Vert x^{\FF} - x' \Vert 
= O(\sqrt{T}(\log K)^{\frac{1}{2}})\label{eq:ff+:x1},
\end{align}
where the second inequality is due to the Cauchy-Schwartz inequality.

Conditioned on $Z$, we have
$\sum_{j=1}^n  r_j z^*_j - \sum_{j=1}^n  r_j x^{\FF}_j = O(\sqrt{T}(\log K)^{\frac{1}{2}})$, which in combination with \eqref{eq:ff+:x1} shows 
$\sum_{j=1}^n  r_j z^*_j - \sum_{j=1}^n  r_j x'_j = O(\sqrt{T}(\log K)^{\frac{1}{2}})$. 
Since $x' \in S \cap Q$, by Lemma~\ref{lm:ff+:x}, we have $\Vert x' - z^*\Vert =  O(\sqrt{T}(\log K)^{\frac{1}{2}})$. 
By the triangle inequality, we have 
\[
\Vert x^{\FF} - z^*\Vert \leq 
\Vert x^{\FF} - x' \Vert + \Vert x' - z^*\Vert = O(\sqrt{T}(\log K)^{\frac{1}{2}}).
\]

To summarize, in both cases, we have $\Vert x^{\FF} - z^*\Vert= O(\sqrt{T}(\log K)^{\frac{1}{2}})$.
Proposition \ref{prop:ff+:x} follows since 
$|x_j^{\FF} - z_j^*| \leq \Vert x^{\FF} - z^*\Vert$ for any $j\in[n]$.

\halmos
\end{proof}

Proposition~\ref{prop:ff:x} follows from Proposition~\ref{prop:ff+:x} where $t^*=T$, $K=T$.

\newpage

\section{Proof for Algorithm~\ref{alg:fd}}\label{ap:fd}



\subsection{Theorem~\ref{thm:fd:regret}}
\begin{proof}{Proof of Theorem~\ref{thm:fd:regret}.}

Consider Algorithm~\ref{alg:fd} with technical parameters
\[
\alpha= \frac{3 \log(\log T)}{ 2 \log T}, \;
\beta = \frac{\log(\log T)}{ \log T}, \;
\gamma =  \frac{2\log(\log T)}{3 \log T} ,\;
\varepsilon = \frac{\log(\log T)}{ \log T}.
\]
Let $t_1':= l_1+1$ be the starting time period of phase II, and $t_2'=l_1+l_2+1$ the starting time period of phase III, where $l_1$, $l_2$ denote the lengths of phase I and phase II, respectively.

Let $V^{\AuxHO\text{-}\I}$, $V^{\AuxHO\text{-}\II}$ denote the revenue of the auxiliary algorithms $\AuxHO$-I and $\AuxHO$-II, respectively.
By definition, we have
\begin{align}
    \bE[V^{\HO} - V^{\FD}] &=
    \bE[V^{\HO} - V^{\AuxHO\text{-I}}]
    + \bE[V^{\AuxHO\text{-I}} - {V}^{\AuxHO\text{-II}}]
    + \bE[{V}^{\AuxHO\text{-II}} - V^{\FD}]
\end{align}

By Proposition~\ref{prop:fd:regret}, we know 
\[
\bE[V^{\HO} - V^{\AuxHO\text{-I}}] = O(T^a) = O(T^{\alpha}) = O((\log T)^{\frac{3}{2}}) , 
\]
\[
\bE[V^{\AuxHO\text{-I}} - {V}^{\AuxHO\text{-II}}] = 
O(Te^{-T^{\varepsilon}})=O(T\cdot \frac{1}{T}) = O(1) , 
\]
\[
\bE[{V}^{\AuxHO\text{-II}} - V^{\FD}]=
O(T^{\frac{3}{4}b+\frac{\varepsilon}{2}}) = O(T^{\frac{3}{8}} T^{\frac{3\beta}{4}+\frac{\varepsilon}{2}}) = O(T^{\frac{3}{8}} (\log T)^{\frac{5}{4}}) . 
\]

Therefore, the regret bound of the algorithm is $O(T^{\frac{3}{8}} (\log T)^{\frac{5}{4}}) $. 
\halmos

\end{proof}

\subsection{Proposition~\ref{prop:fd:regret}}

Recall Algorithm~\ref{alg:fd} with  technical parameters $\alpha$, $\beta$, $\gamma$ that 
satisfy $0<\alpha < \frac{1}{2}$,
$\frac{\alpha}{2} \leq \beta < \frac{1}{2}$ and $0 < \gamma < \frac{\alpha}{2}$.
The algorithm calculates $a = \alpha$, $b = \frac{1}{2}+\beta$ and $c = \frac{\alpha}{2} + \gamma$.
Let $t_1':= l_1+1$ be the starting time period of phase II, and $t_2'=l_1+l_2+1$ the starting time period of phase III, where $l_1=\lceil T^a \rceil$, $l_2=\lceil T-T^a-T^b \rceil$ denote the lengths of phase I and phase II, respectively. 


\begin{proof}{Proof of Proposition~\ref{prop:fd:regret}.}

We first show (1) $ \bE[V^{\HO} - V^{\AuxHO\text{-\normalfont{I}}}] = O(T^a)$.

By definition of the algorithm, we have
$ \bE[V^{\HO} - V^{\AuxHO\text{-\normalfont{I}}}] \leq \bE[V^{\HO} - V^{\HO}(t_1')]$, where $V^{\HO}(t_1')$ denotes the value of the hindsight LP over periods $t_1'$ to $T$ as shown in \eqref{eq:HO_t}. 

Since $\Lambda_j(l_1)$ for $j\in[n]$ is a binomial random variable with mean $\lambda_j l_1$, we apply Hoeffding's inequality to obtain
\begin{align*}
    \bE\left[(\Lambda_j(l_1) - 2\lambda_j l_1 )^{+}\right] 
    &= \int_0^{\infty} P\left(\Lambda_j(l_1) - 2\lambda_j l_1 \geq x\right) dx \\
    &= \int_0^{\infty} P\left(\Lambda_j(l_1) - \lambda_j l_1 \geq x+\lambda_j l_1 \right) dx \\
    &\leq \int_0^{\infty} 2 \exp\left(- \frac{2(x+\lambda_j l_1)^2}{l_1} \right) dx \\
    &\leq \int_0^{\infty} 2 \exp\left(- \frac{2(x+\lambda_j l_1) \lambda_j l_1}{l_1} \right) dx \\
    &= \frac{1}{\lambda_j} \exp(-2\lambda_j l_1^2).
\end{align*}

We then have
\begin{align*}
\bE\left[V^{\HO} - V^{\AuxHO\text{-\normalfont{I}}}\right]
&\leq \bE\left[\sum_{j\in[n]} r_j \Lambda_j(l_1)\right] \leq \bE\left[\sum_{j\in[n]} r_j \big(2\lambda_j l_1 + (\Lambda_j(l_1) - 2\lambda_j l_1 )^{+}\big)\right] \\
&\leq \sum_{j\in[n]} r_j \left( 
2\lambda_j l_1 +\frac{1}{\lambda_j}\exp(-2\lambda_j l_1^2) \right) = O(l_1) = O(T^a).
\end{align*}

We next show (2) $\bE[V^{\AuxHO\text{-\normalfont{I}}} - {V}^{\AuxHO\text{-\normalfont{II}}}]
= O( T e^{-T^{\varepsilon}}) $ given $0 < \varepsilon < 2\gamma$, $\varepsilon\leq \beta$ and 
$\varepsilon \geq \frac{\log (\log T)}{\log T}$.

Consider a variant of Algorithm~\ref{alg:fd} that does not include the thresholding conditions (Line 5 to Line 11). We refer to this algorithm as Algorithm~\ref{alg:fd}'. In other words, Algorithm~\ref{alg:fd}' simply runs the Algorithm~\ref{alg:ff} as a subroutine in Phase II.
Let $x_j(t_1',t_2'-1)$ and $x_j'(t_1',t_2'-1)$ denote the number of type-$j$ customers accepted by Algorithm~\ref{alg:fd} and Algorithm~\ref{alg:fd}', respectively, during $t_1',\ldots,t_2'-1$.

Observe that the virtual capacity vector $B'(t)$ defined in phase II describes the capacity updating procedure of Algorithm~\ref{alg:fd}'.
In addition, Algorithm~\ref{alg:fd} and Algorithm~\ref{alg:fd}' have identical dual updates until the subroutine of Algorithm~\ref{alg:ff} in Algorithm~\ref{alg:fd}' stops. 
Since customer types $j\in \calA$ uses an additional amount of capacities in the case of Algorithm~\ref{alg:fd}, we have
\begin{equation}\label{eq:fd:p2:1}
x_j(t_1',t_2'-1) \leq x_j'(t_1',t_2'-1) \text{  for  }
j \in (\calA\cup \calR)^{\complement} 
\end{equation}

Let $\bar{x}_j(t_1',t_2'-1)$ denote the number of type $j$ customers accepted by Algorithm~\ref{alg:fd} during $t_1',\ldots,t_2'-1$ if we were allowed to go over capacity limits $B(t)$. 
We then have 
\begin{equation}\label{eq:fd:p2:2}
x_j(t_1',t_2'-1) \leq \bar{x}_j(t_1',t_2'-1) \text{  for  }
j \in (\calA \cup \calR),
\end{equation}
\begin{equation} \label{eq:fd:p2:3}
    x_j'(t_1',t_2'-1) = \bar{x}_j(t_1',t_2'-1) \text{  for  }
j \in (\calA \cup \calR)^{\complement}.
\end{equation}


Recall that $\bar{z}_j(t_1')$ is the hindsight optimal solution to $\HO(t_1')$, defined in \eqref{eq:HO_t}.
Define events
\begin{equation}\label{eq:fd:p2:event}
    E_j = \big\{ \bar{x}_j(t_1',t_2'-1) \leq \bar{z}_j(t_1') \leq \bar{x}_j(t_1',t_2'-1) + \Lambda(t_2',T) \big\}, \;\;
    E = \bigcap_{j\in [n]} E_j.
\end{equation}
The definition of these events is similar to that in \cite{bumpensanti2020re}.

Observe that under event $E$, we have $\bar{x}_j(t_1',t_2'-1)\leq \bar{z}_j(t_1')$ for all $j\in[n]$, and thus 
\[
\sum_{j\in[n]} \bar{x}_j(t_1',t_2'-1)\leq \sum_{j\in[n]} \bar{z}_j(t_1')\leq 
B(t_1') =  \frac{T-l_1}{T} C .
\] 

Combining \eqref{eq:fd:p2:1}, \eqref{eq:fd:p2:2} and \eqref{eq:fd:p2:3}, we have
\begin{equation}
    \sum_{j\in[n]} x_j(t_1',t_2'-1)\leq 
    \sum_{j\in[n]} \bar{x}_j(t_1',t_2'-1)\leq B(t_1'),
\end{equation}
which means Algorithm~\ref{alg:fd} will not stop before time $t_2$ due to insufficient capacity $B(t)$. 
Therefore, conditioned on event $E$, we have 
$x_j(t_1',t_2'-1) = \bar{x}_j(t_1',t_2'-1)$ for all $j\in[n]$,
and hence
\begin{equation}
x_j(t_1',t_2'-1) \leq \bar{z}_j(t_1') \leq x_j(t_1',t_2'-1) + \Lambda(t_2',T) . 
\end{equation}
This means the decision-maker would still be able to achieve the hindsight optimum of $\HO(t_1')$ if she follows the decision of Algorithm~\ref{alg:fd} in periods $t_1',\ldots,t_2'-1$ and then accepts $\bar{z}_j(t_1') - x_j(t_1',t_2'-1)$ type-$j$ customers during periods $t_2',\ldots,T$.

Apply Hoeffding's inequality to bound $\left(\Lambda_j(T) - \Lambda_j(l_1) - 2\lambda_j(T-l_1) \right)^{+}$. We have
\begin{align*}
\bE\big[\Lambda_j(T) - \Lambda_j(l_1) & - 2\lambda_j(T-t_1) )^{+}\big] 
= \int_0^{\infty} P\left(\Lambda_j(T) - \Lambda_j(l_1) - 2\lambda_j(T-l_1) \geq x\right) dx \\
&= \int_0^{\infty} P\left(\Lambda_j(T) - \Lambda_j(l_1) - \lambda_j(T-l_1) \geq x+\lambda_j(T-l_1) \right) dx \\
&\leq \int_0^{\infty} 2 \exp\left(- \frac{2(x+\lambda_j(T-l_1) )^2}{T - l_1} \right) dx \\
&\leq \int_0^{\infty} 2 \exp\left(- \frac{2(x+\lambda_j(T-l_1)  )\lambda_j(T-l_1) }{T-l_1} \right) dx \\
&= \frac{1}{\lambda_j} \exp(-2\lambda_j (T-l_1)^2).
\end{align*}

Thus, we have
\begin{align*}
\bE[V^{\AuxHO\text{-}\I} &- V^{\AuxHO\text{-}\II} ] 
\leq \bE\left[ \sum_{j=1}^n r_j (\Lambda_j(T) - \Lambda_j(l_1)) \mid E^\complement \right] P(E^\complement) \\
&\leq \bE\left[ \sum_{j=1}^n r_j \big(2\lambda_j (T-l_1) + (\Lambda_j(T) - \Lambda_j(l_1) + 
    2\lambda_j (T-l_1))^{+}\big) \mid E^\complement \right] P(E^\complement) \\
&\leq \sum_{j=1}^n r_j \bigg( 2\lambda_j (T-l_1) P(E^\complement) + \bE\left[ (\Lambda_j(T) - \Lambda_j(l_1) + 
    2\lambda_j (T-l_1))^{+} \mid E^\complement \right] P(E^\complement) \bigg) \\
&\leq \sum_{j=1}^n r_j \bigg( 2\lambda_j (T-l_1) P(E^\complement) + \frac{1}{\lambda_j} e^{-2\lambda_j^2 (T-l_1)^2} \bigg)
\end{align*}


By Lemma~\ref{lm:fd:p2:regret}, we have 
$P(E^{\complement}) = O(\exp(-T^\varepsilon))$, and therefore,
\begin{equation}
\bE[V^{\AuxHO\text{-\normalfont{I}}} - {V}^{\AuxHO\text{-\normalfont{II}}}]
\leq \sum_{j=1}^n r_j 
\bigg( 2\lambda_j (T-l_1) \, 
O\big( e^{-T^\varepsilon} \big) + \frac{1}{\lambda_j} e^{-2\lambda_j^2 (T-l_1)^2} \bigg) = O(T e^{-T^\varepsilon}).
\end{equation}

We last show (3)  $\bE[V^{\AuxHO\text{-\normalfont{II}}} - V^{\FD}] = O\big(T^{\frac{3b}{4} + \frac{\varepsilon}{2}} \big)$ 
given $0 < \varepsilon < 2\gamma$, $\varepsilon\leq \beta$ and $\varepsilon \geq \frac{\log (\log T)}{\log T}$.

Observe that Algorithm~\ref{alg:fd} runs Algorithm~\ref{alg:ff} as a subroutine using a virtual copy of initial capacity $B''(t_2')=\max\{ T^{3b/4}, 
\min\{B_i(t_2'), \bar{a} T^b\}\}$. In particular, the virtual capacity inflates any resource with small capacity to $T^{3b/4} $, and truncates any resource with large capacity to $ \bar{a} T^b$. By definition of $\bar{a}$, we know $ \bar{a} T^b$ is an upper bound on the units of resource consumption over the length of phase III.
Define
\begin{align}
\begin{split}
\tilde{V}^{\text{DLP}}(t) = \max_x\; & (T-t+1) \sum_{j=1}^n r_j x_j \\
\text{s.t.} \; & \sum_{j=1}^n A_j x_j \leq \frac{B''(t)}{T-t+1},\\
& 0 \leq x_j \leq \lambda_j, \; \forall j \in [n].
\end{split}
\end{align}
Compare $\tilde{V}^{\text{DLP}}(t)$ with $V^\DLP(t_2')$ as defined in \eqref{eq:DLP_t}. We know 
\begin{align}\label{eq:phase3a2}
     V^\DLP(t_2') 
     \leq \tilde{V}^{\text{DLP}}(t_2').
\end{align}

Now consider those resources with inflated virtual capacities.
Recall that $\bar{\alpha}_i$ denotes the upper bound on the revenue that can be achieved from one unit of resource $i$.
Let $v^{\FD}$ denote the actual revenue of Algorithm~\ref{alg:fd} in phase III, namely $v^{\FD}=\sum_{t=t_2'}^{T} r_{j(t)} z_t$ and $\tilde{v}^{\FD}$ the revenue of the Algorithm~\ref{alg:ff} subroutine based on virtual capacity $B''(t_2')$, namely $v^{\FD}=\sum_{t=t_2'}^{T} r_{j(t)} y_t$.
We have
\begin{align}\label{eq:phase3a1}
 \tilde{v}^{\FD} 
 \leq v^{\FD}
     + T^{3b/4}\,m\, \max_{i\in[m]} \bar{\alpha}_i.
\end{align}
Combining \eqref{eq:phase3a2} and \eqref{eq:phase3a1}, we have
\begin{align}\label{eq:phase3a3}
    V^\DLP(t_2') - v^{\FD}
     \leq \tilde{V}^{\text{DLP}}(t_2') - \tilde{v}^{\FD} 
     + T^{3b/4}\,m\, \max_{i\in[m]} \bar{\alpha}_i.
\end{align}

Notice that $\tilde{V}^{\text{DLP}}(t_2') - \tilde{v}^{\FD} $ is the regret of Algorithm~\ref{alg:ff} in phase III based on the virtual capacity $B''(t_2')$, and we can apply Proposition~\ref{prop:ff+:regret} to upper bound this regret.
In particular, the notations in Proposition~\ref{prop:ff+:regret} become $C_\text{max}=T^b$, $C_\text{min}=T^{3b/4}$, and thus 
$\frac{C_\text{max}}{C_\text{min}}=T^{b/4}$. 
Now take $K = e^{T^{\varepsilon}}$. We have $\log K =  T^{\varepsilon}$. Given $\varepsilon \geq \frac{\log(\log T)}{\log T}$, we know $\frac{1}{K} = e^{-T^{\varepsilon}} \leq \frac{1}{T}$.


By Proposition~\ref{prop:ff+:regret}, we have 
\begin{align}\label{eq:phase3a4}
    \tilde{V}^{\text{DLP}}(t_2') - v^{\FD} 
    &\leq \left(\rho_2 \frac{C_\text{max}}{C_\text{min}} + \rho_3\right) \sqrt{T}(\log K)^\frac{1}{2} + \left(\rho_1  \frac{C_\text{max}^2}{C_\text{min} T} + \rho_0 \frac{C_\text{max}}{C_\text{min}} \right) \sqrt{T} + \rho_4 \frac{C_\text{max}}{C_\text{min}}\nonumber\\
    &= \left(\rho_2 T^{\frac{b}{4}} + \rho_3\right) \sqrt{T^b}(\log K)^\frac{1}{2} + \left(\rho_1  T^{2b} T^{-\frac{3b}{4}-1} + \rho_0 T^{\frac{b}{4}} \right) \sqrt{T^b} + \rho_4 T^{\frac{b}{4}} \nonumber\\
    &=\left(\rho_2 T^{\frac{b}{4}} + \rho_3\right) T^{\frac{b}{2} + \frac{\varepsilon}{2}} + \left(\rho_1 T^{\frac{5b}{4}-1} + \rho_0 T^{\frac{b}{4}} \right) T^{\frac{b}{2}}
    + \rho_4 T^{\frac{b}{4}}, 
\end{align}
with probability at least $1-e^{-T^{\varepsilon}}=1-\frac{1}{K}$.

For simplicity of exposition, let $v_1:={V}^{\text{DLP}}(t_2') $
and $v_2:=\tilde{V}^{\text{DLP}}(t_2')$. 

Define event $F=\big\{v_2 - v^{\FD}\leq\left(\rho_2 \frac{C_\text{max}}{C_\text{min}} + \rho_3\right) \sqrt{T}(\log K)^\frac{1}{2} + \left(\rho_1  \frac{C_\text{max}^2}{C_\text{min} T} + \rho_0 \frac{C_\text{max}}{C_\text{min}} \right) \sqrt{T} + \rho_4 \frac{C_\text{max}}{C_\text{min}} \big\}$.
By definition of auxiliary algorithm $\AuxHO$-II, we have
\begin{align}
\bE\left[V^{\AuxHO\text{-}\II} - V^{\FD} \right]&\leq \bE\left[v_1 - v^{\FD} \right] \nonumber \\
&=\bE[v_1 - v^{\FD} \mid F^\complement] P(F^\complement)
+\bE[v_1 - v^{\FD} \mid F] P(F) \nonumber\\
&\leq \bE\left[ \sum_{j=1}^n r_j \Lambda_j(t_2',T) \right] P(F^\complement)+\bE[v_1 - v^{\FD} \mid F] P(F).
\end{align}

In particular, we have
\begin{align*}
\bE\left[ \sum_{j=1}^n r_j \Lambda_j(t_2',T) \right]
&\leq \bE\left[ \sum_{j=1}^n r_j \big(2\lambda_j (T-t_2'+1) + (\Lambda_j(t_2',T) + 2\lambda_j (T-t_2'+1))^{+}\big) \right] \\
&\leq \sum_{j=1}^n r_j \bigg( 2\lambda_j T^b + \frac{1}{\lambda_j} e^{-2\lambda_j^2 T^{2b}} \bigg) ,
\end{align*}
where the second term is due to the following 
\begin{align*}
 \bE[(\Lambda_j(t_2',T) & - 2\lambda_j(T-t_2'+1) )^{+}]
= \int_0^{\infty} P\left(\Lambda_j(t_2',T) - 2\lambda_j(T-t_2'+1) \geq x\right) dx \\
&= \int_0^{\infty} P\left(\Lambda_j(t_2',T) - \lambda_j(T-t_2'+1) \geq x+ \lambda_j(T-t_2'+1)\right) dx \\
&\leq \int_0^{\infty} 2 \exp\left(- \frac{2(x+\lambda_j(T-t_2'+1) )^2}{T - t_2'+1} \right) dx \\
&\leq \int_0^{\infty} 2 \exp\left(- \frac{2(x+\lambda_j(T-t_2'+1) )\lambda_j(T-t_2'+1) }{T - t_2'+1} \right) dx\\ 
&= \frac{1}{\lambda_j} \exp(-2\lambda_j (T-t_2'+1)^2) \\
&= \frac{1}{\lambda_j} \exp(-2\lambda_j T^{2b}).
\end{align*}

Additionally, conditioned on event $F$, by \eqref{eq:phase3a3} and \eqref{eq:phase3a4}, we have 
\begin{equation}\label{eq:phase3a5}
    v_1 - v^{\FD} \leq \left(\rho_2 T^{\frac{b}{4}} + \rho_3\right) T^{\frac{b}{2} + \frac{\varepsilon}{2}} + \left(\rho_1 T^{\frac{5b}{4}-1} + \rho_0 T^{\frac{b}{4}} \right) T^{\frac{b}{2}}
    + \rho_4 T^{\frac{b}{4}} + \rho_5 T^{\frac{3}{4}b}
    =O(T^{\frac{3}{4}b+\frac{\varepsilon}{2}}),
\end{equation}
where 
$\rho_5:=m\, \max_{i\in[m]} \bar{\alpha}_i$. The equality follows given that $\varepsilon > 0$ and $b=\frac{1}{2} + \beta < 1$ by definition.

Combining the results above, we have
\begin{align*}
\bE\left[V^{\AuxHO\text{-}\II} - V^{\FD} \right] 
&\leq \bE\left[ \sum_{j=1}^n r_j \Lambda_j(t_2',T) \right] P(F^\complement)+\bE[v_1 - v^{\FD} \mid F].\\
&\leq \sum_{j=1}^n r_j \bigg( 2\lambda_j T^b + \frac{1}{\lambda_j} e^{-2\lambda_j^2 T^{2b}} \bigg) e^{-T^{\varepsilon}} +
O(T^{\frac{3b}{4}+\frac{\varepsilon}{2}})\\
&= O(T^{\frac{3}{4}b+\frac{\varepsilon}{2}}).
\end{align*}

This completes the proof. \halmos

\end{proof}

\subsection{Lemma~\ref{lm:fd:p2:x}}

\begin{lemma}\label{lm:fd:p2:x}
Let $x_j^*(t)$ for $j\in[n]$ be the optimal solution to $\DLP(t)$ in \eqref{eq:DLP_t}.
Define variable 
\begin{equation}
\Gamma(t,T) = \kappa \sum_{j: x_j^* = \lambda_j} \Big|\Lambda_j(t,T) - \lambda_j(T-t+1) \Big| ,
\end{equation}
where $\kappa$ is is a constant whose value is determined by the BOM matrix $A= (a_{ij} )_{i\in[m],j\in[n]}$. 
Recall $\bar{z}_j(t)$ for $j\in[n]$ is the optimal solution to $\HO(t)$ as defined in \eqref{eq:HO_t}.
We have
\begin{equation}\label{eq:fd:x}
    \bar{z}_j(t ) \in 
    \Big[ (T-t+1) x_j^*(t) - \Gamma(t,T) , (T-t+1) x_j^*(t) + \Gamma(t,T)  \Big].
\end{equation}

\end{lemma}

\begin{remark}
More specifically, $\kappa$ is the maximum absolute value of
the elements in the inverses of all invertible submatrices of the BOM matrix $A$. In a special case when all entries of $A$ are either $0$ or $1$, we have $\kappa \leq \max\{1, \min\{m,n\}-1\}$.
\end{remark}

\begin{proof}{Proof of Lemma~\ref{lm:fd:p2:x}.}
The Lemma follows theorem 4.2 in Reiman and Wang (2008).
\end{proof}

\subsection{Lemma~\ref{lm:fd:p2:regret}}

Recall Algorithm~\ref{alg:fd} with start time $t_1=1$, end time $t_2 = T$, initial capacity vector $C$,
and technical parameters $\alpha $, $\beta$, $\gamma$ that satisfy
$0<\alpha < \frac{1}{2}$,
$\alpha \leq \beta < \frac{1}{2}$ and $0 < \gamma < \frac{\alpha}{2}$.
The algorithm calculates $a = \alpha$, $b = \frac{1}{2}+\beta$ and $c = \frac{\alpha}{2} + \gamma$.
Let $t_1':= l_1+1$ be the starting time period of phase II, and $t_2'=l_1+l_2+1$ the starting time period of phase III, where $l_1=\lceil T^a \rceil$, $l_2=\lceil T-T^a-T^b \rceil$ denote the lengths of phase I and phase II, respectively.

Recall that $\bar{x}_j(t_1',t_2'-1)$ is the number of type $j$ customers accepted by Algorithm~\ref{alg:fd} over periods $t_1', \ldots, t_2'-1$ if we were allowed to go over the real capacity limits $B(t)$, $\bar{z}_j(t_1')$ for $j\in[n]$ is the hindsight optimal solution to $\HO(t_1')$, and $x_j^*$ for $j\in[n]$ is the optimal solution to the $\DLP$ in \eqref{eq:DLP1}. 
Under Assumption~\ref{asmp:1}, by Lemma~\ref{lm:asmp:1}, we know \eqref{eq:DLP1} has a unique optimal solution, which is equivalent to the optimal solution of $\DLP(t_1')$, and thus
we know $x_j^* = x_j^*(t_1')$ for $j\in[n]$.

\begin{lemma}\label{lm:fd:p2:regret}

The probability of event $E$ defined in \eqref{eq:fd:p2:event} satisfies $$P(E^\complement) = O(e^{-T^\varepsilon})$$ for any $\varepsilon$ such that $0 < \varepsilon < 2\gamma$
and $\varepsilon \geq \frac{\log (\log T)}{\log T}$.
\end{lemma}

\begin{proof}{Proof.}

By union bound, we have $P(E^\complement) \leq \sum_{j\in[n]} P(E_j^\complement)$. Thus, it suffices to show 
\begin{equation}\label{eq:event_master}
    P(E_j) = 1-O(e^{-T^\varepsilon})
    \text{  or  } P(E_j^\complement) = O(e^{-T^\varepsilon}) \text{ for all } j \in [n].
\end{equation}


Define events
\begin{equation}
 E_{1,j} = \big\{ \bar{z}(t_1')_j - \bar{x}_j(t_1',t_2'-1) \geq 0 \big\} , \; \forall \; j\in[n] 
\end{equation}
\begin{equation}
 E_{2,j} = \big\{ \bar{x}_j(t_1',t_2'-1) + \Lambda_j(t_2',T) - \bar{z}_j(t_1') \geq 0 \big\} , \; \forall \; j\in[n].
\end{equation}
Thus, we have $E_j = E_{1,j} \cap E_{2,j}$ for all $j\in [n]$.

Let $x_j(1,t_1'-1)$ denote the number of type-$j$ customers accepted by Algorithm~\ref{alg:fd} during periods $1,\ldots,t_1'-1$. Recall that $l_1=t_1'-1$. Apply Proposition~\ref{prop:ff+:x} to bound $|x_j(1,t_1'-1) - l_1 x^*_j|$ given $K=\Omega(T)$, and we obtain 
\begin{equation}
    |x_j(1,t_1'-1) - l_1 x^*_j| = O(\sqrt{l_1}(\log K)^{\frac{1}{2}})
\end{equation}
with probability at least $1-O(\frac{1}{K})$.

In addition, by definition, we know  $\bar{x}_j(t_1',t_2'-1)$ for $j\in(\calR \cup \calA)^\complement$
are decisions of the Algorithm~\ref{alg:ff} subroutine in phase II.
Apply Proposition~\ref{prop:ff+:x} to bound $|\bar{x}_j(t_1',t_2'-1) - l_2 x^*_j|$ given $K=\Omega(T)$, and we obtain 
\begin{equation}
    |\bar{x}_j(t_1',t_2'-1) - l_2 x^*_j| = O(\sqrt{l_2}(\log K)^{\frac{1}{2}})
\end{equation} 
with probability at least $1-O(\frac{1}{K})$.

Given $0 < \varepsilon < 2\gamma $ and 
$\varepsilon \geq \frac{\log (\log T)}{\log T}$, set $ K = \exp(T^{\varepsilon})$. We then have $K=T$ and $ \frac{1}{K} = \frac{1}{T} $.

In the following proof, we will show \eqref{eq:event_master} in three cases:
$$\text{1)}\, x_j^{*}=0 , \;\;
\text{2)} \, x_j^{*}=\lambda_j , \;\; 
\text{and 3)}\, x_j^{*} \in (0,\lambda_j).$$

First, we consider case 1) $x_j^{*}=0$.

By Proposition~\ref{prop:ff+:x}, we know \[
x_j(1,t_1'-1) = O(\sqrt{l_1 \log K})
= O(T^{\frac{a}{2}+\frac{\varepsilon}{2}})
= O(T^{\frac{\alpha}{2}+\gamma})=O(T^c)
\]
with probability at least $1-O(1/K) = 1-O(e^{-T^{\varepsilon}})$,
where the last inequality follows since $\varepsilon < 2\gamma$. This implies $P(j \in \calR) = 1-O(e^{-T^\varepsilon})$.

Conditioned on $j\in R$, we have $\bar{z}_j(t_1') \geq \bar{x}_j(t_1',t_2'-1)=0$, and thus $P(E_{1,j} \mid j\in R) = 1$.

Now consider event $E_{2,j}$. Define variables
\begin{equation}
    \Delta_j (t ,T) = \Lambda_j(t,T) - \lambda_j(T-t), \; \forall \; j\in[n],
\end{equation}

We know by \eqref{eq:fd:x} in Lemma~\ref{lm:fd:p2:x} that
$\bar{z}_j(t_1') \leq (T-l_1)x_j^*+\Gamma(t_1',T) = \Gamma(t_1',T)$. 
Thus, we have 
\begin{align*}
P(E_{2,j}^{\complement} \mid j\in \calR) &=
P\big(\bar{z}_j(t_1') > \Lambda_j(t_2',T)  \big) \\
&\leq P\big(\Gamma(t_1',T) > \Lambda_j(t_2',T)  \big)\\
&\leq P\big(\Gamma(t_1',T) + |\Delta_j(t_2',T)| > \lambda_j (T-t_2'+1)\big) \\
&\leq P\big(\Gamma(t_1',T) > \frac{1}{2}\lambda_j(T-t_2'+1)) + P\big(|\Delta_j(t_2',T)| > \frac{1}{2}\lambda_j(T-t_2'+1)\big),
\end{align*}
where the second inequality follows since $\Lambda_j(t_2',T)  \geq\lambda_j(T-t_2'+1) - |\Delta_j(t_2',T)| $, and the last inequality follows by union bound and the observation that when
$\{\Gamma(t_1',T) + |\Delta_j(t_2',T)| > \lambda_j(T-t_2'+1)\}$ happens, at least one of the following two events 
\[\{\Gamma(t_1',T) > \frac{1}{2}\lambda_j(T-t_2'+1)\} 
\text{  and  }
\{|\Delta_j(t_2',T)| > \frac{1}{2}\lambda_j(T-t_2'+1)\}
\] must happen.

Let $J_{\lambda}=\{j: x_j^* =\lambda_j\}$. By definition, 
\begin{align}
    P\big(\Gamma(t_1',T) > \frac{1}{2} & \lambda_j(T-t_2'+1) \big) \nonumber
    = P\left(\kappa \sum_{j\in J_{\lambda}} \Big|\Lambda_j(t_1',T) - \lambda_j(T-t_1'+1) \Big| > \frac{1}{2}\lambda_j(T-t_2'+1) \right)  \nonumber\\
    \leq\,& P\left(\kappa \sum_{j\in J_{\lambda}} \Big|\Lambda_j(t_1',T) - \lambda_j(T-t_1'+1) \Big| > \frac{1}{2}\lambda_j(T-t_2'+1)  \mathbf{1}(|J_{\lambda}|>0)\right)  \nonumber\\
    \leq\,& \sum_{j\in J_{\lambda}} 
    P\left(\kappa\,  \Big|\Lambda_j(t_1',T) - \lambda_j(T-t_1'+1) \Big| > \frac{1}{2|J_{\lambda}|}\lambda_j(T-t_2'+1)
    \mathbf{1}(|J_{\lambda}|>0)\right)  \nonumber\\
    \leq\,& \sum_{j\in J_{\lambda}} 
    P\left(\Big|\Lambda_j(t_1',T) - \lambda_j(T-t_1'+1) \Big| > \frac{\lambda_j(T-t_2'+1)}{2\kappa |J_{\lambda}|+1} \right).
\end{align}

By Hoeffding's inequality, we have
\begin{align}
P \Big(\Big|\Lambda_j(t_1',T)  & - \lambda_j(T-t_1'+1) \Big| > \frac{\lambda_j(T-t_2'+1)}{2\kappa |J_{\lambda}|+1}  \Big) \nonumber\\
\leq\, & 2\exp\left(- \frac{2\lambda_j^2(T-t_2'+1)^2}{(2\kappa |J_{\lambda}|+1)^2 (T-t_1)}\right)
= O(\exp(-T^{2b-1})) 
\end{align}

Thus, we obtain
\begin{equation}\label{eq:fd:case1:1}
    P\big(\Gamma(t_1',T) > \frac{1}{2}\lambda_j(T-t_2'+1) \big) = O(\exp(-T^{2b-1}))
\end{equation}

In addition, by Hoeffding's inequality, we have
\begin{align}\label{eq:fd:case1:2}
    P\left(|\Delta_j(t_2',T)| > \frac{1}{2}\lambda_j(T-t_2'+1)\right) 
    \leq 2\exp\left(-\frac{\lambda_j^2(T-t_2'+1)^2}{2(T-t_2'+1)}\right) = O(\exp(-T^{b}))
\end{align}

Combining \eqref{eq:fd:case1:1} and \eqref{eq:fd:case1:2}, we have
\begin{equation}
P(E_{2,j}^{\complement} \mid j\in R) = O(\exp(-T^{2b-1})) = O(-e^{T^\beta}).
\end{equation}

Therefore, in case 1), we have
\begin{align}
    P(E_{j}) &\geq P(E_{1,j} \cap E_{2,j} \mid j\in \calR) \cdot P(j\in \calR) \nonumber\\
    &\geq \big(1 - P(E^c_{1,j} \mid j\in \calR) 
    - P(E^c_{2,j} \mid j\in \calR) \big) \cdot P(j\in \calR) \nonumber\\
    &= 1- O(e^{-T^\beta}) - O(e^{-T^\varepsilon})\nonumber\\
    &= 1- O(e^{-T^\varepsilon}) ,
    \label{eq:fd:case1}
\end{align}
where the last line follows since $0 < \varepsilon \leq \beta$.

Next, we consider case 2) $x_j^{*}=\lambda_j$. 

By Proposition~\ref{prop:ff+:x}, we know with probability at least $1-O(1/K) =1-O(e^{-T^\varepsilon})$ that 
$$x_j(1,t_1'-1) = \lambda_j l_1-O(\sqrt{l_1}(\log K)^{\frac{1}{2}}) = O(T^{a} - T^{\frac{a}{2}+\frac{\varepsilon}{2}})
= O(T^{a} - T^{c}),$$
where the last equality follows since $\varepsilon < 2\gamma$. This implies $P(j \in A) = 1-O(e^{-T^\varepsilon})$.

Conditioned on $j\in \calA$, we have
$\bar{x}_j(t_1',t_2'-1)=\Lambda(t_1',t_2'-1)$, 
and $\bar{z}_j(t_1') \leq \Lambda(t_1',T) = \bar{x}_j(t_1',t_2'-1) + \Lambda(t_2',T)$. 
Thus, $P(E_{2,j} \mid j\in A) = 1$.

Now consider event $E_{1,j}$. 
We know by \eqref{eq:fd:x} in Lemma~\ref{lm:fd:p2:x} that
$\bar{z}_j(t_1') \geq (T-l_1)x_j^*-\Gamma(t_1',T)$. 
Thus, we have 
\begin{align*}
P(E_{1,j}^{\complement} \mid j\in \calA) &=
P\big(\bar{z}_j(t_1') < \Lambda_j(t_1',t_2'-1) \big) \\
&\leq P\big(\lambda_j(T-t_1'+1)-\Gamma(t_1',T) < \Lambda_j(t_1',t_2'-1) \big)\\
&\leq P\big(\lambda_j(T-t_1'+1)-\Gamma(t_1',T) < \lambda_j l_2 + |\Delta_j(t_1',t_2'-1)|\big)\\
&= P\big(\Gamma(t_1',T) + |\Delta_j(t_1',t_2')|
> \lambda_j(T-t_2'+1) \big) \\
&\leq P\big(\Gamma(t_1',T)
> \frac{1}{2}\lambda_j(T-t_2'+1) \big) 
+ P\big(|\Delta_j(t_1',t_2'-1)|
> \frac{1}{2}\lambda_j(T-t_2'+1) \big) \\
& = O(\exp(-T^{2b-1})) = O(-e^{T^\beta}),
\end{align*}
where the last line follows because
\begin{equation}
P\big(\Gamma(t_1',T) > \frac{1}{2}\lambda_j(T-t_2'+1) \big) 
= O(\exp(-T^{2b-1})),
\end{equation}
\begin{equation}
P\big(|\Delta_j(t_1',t_2'-1)| > \frac{1}{2}\lambda_j(T-t_2'+1) \big)
\leq 2 \exp\left(-\frac{\lambda_j^2(T-t_2'+1)^2}{2 l_2}\right) = O(\exp(-T^{2b-1})).
\end{equation}

Therefore, in case 2), we have,
\begin{align}
    P(E_{j}) &\geq P(E_{1,j} \cap E_{2,j} \mid j\in \calA) \cdot P(j\in \calA) \nonumber\\
    &\geq \big(1 - P(E^\complement_{1,j} \mid j\in \calA)
    - P(E^\complement_{2,j} \mid j\in \calA) \big) \cdot P(j\in \calA) \nonumber\\
    &= 1- O(e^{-T^\beta}) - O(e^{-T^\varepsilon}) \nonumber \\
    &= 1 - O(e^{-T^\varepsilon}) .
    \label{eq:fd:case2}
\end{align}










Last, we consider case 3) $x_j^{*} \in (0,\lambda_j)$. We have 
\begin{align}\label{eq:fd:case3:0}
    P(E_j) = &P(E_j \mid j\in \calR) P(j \in \calR) + P(E_j \mid j\in \calA) P(j \in \calA) \nonumber\\
    &+P(E_j \mid j\in (\calA\cup \calR)^\complement) P(j \in (\calA\cup \calR)^\complement).
\end{align}

We first show that $P(j\in \calR) = O(e^{-T^\varepsilon})$ and $P(j\in \calA) = O(e^{-T^\varepsilon})$.

By Proposition~\ref{prop:ff+:x}, we know with probability at least $1-O(1/K) =1-O(e^{-T^\varepsilon})$ that
$$x_j(1,t_1'-1) = l_1 x_j^{*} - O(\sqrt{l_1 \log K})
= O(T^a) - O(T^{\frac{a}{2}+\frac{\varepsilon}{2}})
> \lambda_{j} T^{\frac{a}{2}+\gamma}= \lambda_{j} T^c,$$
where the last inequality follows since $0<\varepsilon < 2\gamma < \alpha$. 
This implies 
\begin{equation}\label{eq:fd:case3:1}
    P(j\in R) =P(x_j(1,t_1'-1) < \lambda_j T^{c}) = O(e^{-T^\varepsilon}) . 
\end{equation}

In addition, we know with probability at least $1-O(1/K) = 1-O(e^{-T^\varepsilon})$ that
$$x_j(1,t_1'-1) = l_1 x_j^{*} + O(\sqrt{l_1 \log K}) = x_j^{*} T^a - O(T^{\frac{a}{2}+\frac{\varepsilon}{2}})
< \lambda_j (T^{a} - T^{c}),$$
where the last inequality follows since 
$a = \alpha > c = \frac{\alpha}{2} + \gamma >
\frac{\alpha}{2} + \frac{\varepsilon}{2}$
and 
$(\lambda_j -x_j^{*} )T^{a} > O(T^c) - O(T^{\frac{a}{2}+\frac{\varepsilon}{2}}) > 0$.
This implies 
\begin{equation}\label{eq:fd:case3:2}
    P(j\in \calA) =P(x_j(1,t_1'-1) > \lambda_j (T^{a} - T^{c})) = O(e^{-T^\varepsilon}) . 
\end{equation}

Combining \eqref{eq:fd:case3:1} and \eqref{eq:fd:case3:2}, we have $P(j\in(\calA\cup \calR)^\complement) = 1-O(e^{-T^\varepsilon})$.

Now consider event $E_{1,j}$ conditioned on $j\in(\calA \cup \calR)^\complement$.

We know by \eqref{eq:fd:x} in Lemma~\ref{lm:fd:p2:x} that
$\bar{z}_j(t_1') \geq (T-l_1)x_j^*-\Gamma(t_1',T)$.
Thus, we have 
\begin{equation}
P(E_{1,j}^\complement) =
P\big(\bar{z}_j(t_1') < \bar{x}(t_1', t_2'-1)\big)
\leq P\big( (T-l_1)x_j^*-\Gamma(t_1',T) < \bar{x}(t_1', t_2'-1) \big)
\end{equation}

Observe that the event 
$\{(T-l_1)x_j^*-\Gamma(t_1',T) < \bar{x}_j (t_1', t_2'-1)\}$ is equivalent to 
\begin{equation}
    \left\{ \bar{x}_j(t_1',t_2'-1) - l_2 x_j^* 
    + \frac{1}{l_1}(T-t_2'+1) \bigg( x_j(1,t_1'-1) - l_1 x_j^* \bigg) + \Gamma(t_1',T)
    > \frac{1}{l_1}(T - t_2'+1) x_j(1,t_1'-1) \right\},
\end{equation}

Therefore, when $E_{1,j}^c$ happens, at least one of the following three events must happen:
\begin{equation}
    E_{1,j}^{(a)}: \left\{ \bar{x}_j(t_1',t_2'-1) - l_2 x_j^* 
    > \frac{1}{3} \frac{1}{l_1}(T - t_2'+1) x_j(1,t_1'-1) \right\},
\end{equation}
\begin{equation}
   E_{1,j}^{(b)}: \left\{ \frac{1}{l_1}(T - t_2'+1) \bigg( x_j(1,t_1'-1) - l_1 x_j^* \bigg)  > \frac{1}{3} \frac{1}{l_1}(T - t_2'+1) x_j(1,t_1'-1)  \right\},
\end{equation}
\begin{equation}
   E_{1,j}^{(c)}: \left\{ \Gamma(t_1',T) > \frac{1}{3} \frac{1}{l_1}(T - t_2' +1) x_j(1,t_1'-1) \right\}.
\end{equation}

By Proposition~\ref{prop:ff+:x}, we know that with probability at least $1-O(1/K) =1-O(e^{T^{\varepsilon}})$,
$$\bar{x}(t_1',t_2'-1) - l_2 x_j^* = O\bigg(\sqrt{l_2} (\log K)^{\frac{1}{2}} \bigg)
= O(T^{\frac{1}{2}+\frac{\varepsilon}{2}})
= O(T^{\frac{1}{2} + \beta -\frac{\alpha}{2} + \gamma}) = O(T^{b-a+c}) , $$
where the last inequality follows since
$0 < \varepsilon < 2\gamma$ and $\beta \geq \frac{\alpha}{2} $.

Since $x_j(1,t_1'-1) \geq \lambda_j T^{c}$ given $j\in(\calA \cup \calR)^\complement$, we know
\[\frac{1}{l_1}(T - t_2'+1) x_j(1,t_1'-1) \geq T^{b-a+c}, \] and therefore, we have
\begin{equation}\label{eq:fd:case3:3}
P\big(E_{1,j}^{(a)}\big) = O\bigg(\bar{x}(t_1',t_2'-1) - l_2 x_j^* > \frac{1}{3}\lambda_j T^{b-a+c} \bigg) 
= O(e^{-T^\varepsilon}).
\end{equation}

In addition, we know with probability at least $1-O(1/K)=1-O(e^{T^{\varepsilon}})$ that 
$$x(1,t_1'-1) - l_1 x_j^* = O\big(\sqrt{l_1} (\log K)^\frac{1}{2} \big)  = O(T^{\frac{a}{2}+\frac{\varepsilon}{2}}) 
= O(T^{\frac{a}{2}+\gamma}) 
= O(T^c).$$

Thus, we have 
\begin{equation}\label{eq:fd:case3:4}
P\big(E_{1,j}^{(b)}\big) = O\bigg(x(1,t_1'-1) - l_1 x_j^* > \frac{1}{3}\lambda_j T^{c} \bigg) = O(e^{-T^\varepsilon}).
\end{equation}

By Hoeffding's inequality, we have 
\begin{align}
    P\big(E_{1,j}^{(c)}\big) &\leq
    P\big(\Gamma(t_1',T) > \frac{1}{3} \lambda_j T^{b-a+c}\big) \nonumber\\
    &\leq \sum_{j\in J_{\lambda}} 
    P\left(\Big|\Lambda_j(t_1',T) - \lambda_j(T-l_1) \Big| > \frac{\lambda_j T^{b-a+c}}{3\kappa |J_{\lambda}|+1} \right) \nonumber\\
    &\leq 2|J_{\lambda}|\exp\bigg(-
    \frac{2\lambda_j^2 \cdot T^{2(b-a+c)}}{(3\kappa |J_{\lambda}|+1)^2 \cdot (T-l_1)}\bigg) \nonumber\\
    &= O(e^{-T^\varepsilon}), \label{eq:fd:case3:5}
\end{align}
where the last inequality follows since 
$2(b-a+c)-1=2\beta - \alpha + 2\gamma \geq 2\gamma > \varepsilon > 0$.

Therefore, combining \eqref{eq:fd:case3:3}, \eqref{eq:fd:case3:4} and \eqref{eq:fd:case3:5}, we have 
\begin{equation}\label{eq:fd:case3:10}
P(E_{1,j}^\complement \mid j\in(\calA\cup  \calR)^\complement) 
\leq P\big(E_{1,j}^{(a)}
\cup E_{1,j}^{(b)}
\cup E_{1,j}^{(c)} \mid j\in(\calA\cup  \calR)^\complement \big)
= O(e^{-T^\varepsilon}).
\end{equation}

Now consider event $E_{2,j}$ conditioned on $j\in(\calA\cup \calR)^\complement$.

We know by \eqref{eq:fd:x} that $\bar{z}_j(t_1') \leq (T-l_1)x_j^*+\Gamma(t_1',T)$.
Thus, we have 
\begin{align}
P(E_{2,j}^{\complement}) &=
P\big(\bar{z}_j(t_1') > \bar{x}(t_1', t_2'-1) + \Lambda(t_2',T) \big) \nonumber\\
&\leq P\big( (T-l_1)x_j^*+\Gamma(t_1',T) > \bar{x}(t_1', t_2'-1) + \Lambda(t_2',T) \big) \\
&\leq P\big( (T-l_1)x_j^*+ \Gamma(t_1',T) > \bar{x}_j(t_1',t_2'-1) - | \Delta_j(t_2',T) | + \lambda_j (T-t_2'+1) \big).
\end{align}

Observe that the event is equivalent to 
\begin{align}
    \bigg\{ l_2 x^*_j &- \bar{x}_j(t_1',t_2'-1)
    + \frac{1}{l_1}(T - t_2'+1)\bigg( l_1x_j^* - x_j(1,t_1'-1) \bigg) \nonumber \\
    &+ \Gamma(t_1',T) + | \Delta_j(t_2',T) |
    > \frac{1}{l_1}(T - t_2'+1)
    \bigg( \lambda_j l_1 - x_j(1,t_1'-1) \bigg) \bigg\}
\end{align}

Therefore, when $E_{2,j}^c$ happens, at least one of the following four events must happen:
\begin{equation}
   E_{2,j}^{(a)} : \; \left\{ 
   l_2 x^*_j - \bar{x}_j(t_1',t_2'-1)
    > \frac{1}{4} \frac{1}{l_1}(T - t_2'+1)
    \bigg( \lambda_j l_1 - x_j(1,t_1'-1) \bigg) \right\},
\end{equation}
\begin{equation}
   E_{2,j}^{(b)} :\; \left\{ \frac{1}{l_1}(T - t_2'+1)
   \bigg( l_1 x_j^* - x(1,t_1'-1) \bigg)  > 
   \frac{1}{4} \frac{1}{l_1}(T - t_2'+1)
    \bigg( \lambda_j l_1 - x_j(1,t_1'-1) \bigg) \right\},
\end{equation}
\begin{equation}
   E_{2,j}^{(c)} :\; \left\{ \Gamma(t_1',T) > \frac{1}{4} \frac{1}{l_1}(T - t_2'+1)
    \bigg( \lambda_j l_1 - x_j(1,t_1'-1) \bigg) \right\}.
\end{equation}
\begin{equation}
   E_{2,j}^{(d)} :\; \left\{ | \Delta_j(t_2',T) | > \frac{1}{4} \frac{1}{l_1}(T - t_2'+1)
    \bigg( \lambda_j l_1 - x_j(1,t_1'-1) \bigg) \right\}.
\end{equation}

By Proposition~\ref{prop:ff+:x}, we know with probability at least $1-O(1/K)=1-O(e^{T^{\varepsilon}})$ that 
$$ l_2 x_j^* - \bar{x}(t_1',t_2'-1)= O\bigg(\sqrt{l_2}(\log K)^{\frac{1}{2}} \bigg)
= O(T^{\frac{1}{2}+\frac{\varepsilon}{2}}).$$

Since $x_j(1,t_1'-1) \leq \lambda_j (T^{a}-T^{c})$ given $j\in(A\cup R)^C$, we know $\lambda_j l_1 - x_j(1,t_1'-1) \geq \lambda_j T^{c}$, and hence
\[\frac{1}{l_1}(T - t_2'+1)
\bigg( \lambda_j l_1 - x_j(1,t_1'-1) \bigg) 
>T^{b-a+c} . \]

Therefore, we have
\begin{equation}\label{eq:fd:case3:6}
P\big(E_{2,j}^{(a)}\big) \leq P\bigg(l_1 x_j^* - \bar{x}(t_1',t_2'-1) > \frac{1}{4}\lambda_j T^{b-a+c} \bigg) = O(e^{-T^\varepsilon}).
\end{equation}

In addition, we know with probability at least $1-O(1/K)=1-O(e^{T^{\varepsilon}})$ that 
$$ l_1 x_j^* - x(1,t_1'-1) = O\big(\sqrt{l_1} (\log K)^{\frac{1}{2}} \big) = O(T^{\frac{a}{2}+\frac{\varepsilon}{2}})
= O(T^c).$$

Thus, we have 
\begin{equation}\label{eq:fd:case3:7}
P\big(E_{2,j}^{(b)}\big) \leq P\bigg(l_1 x_j^* - x(1,t_1'-1) > \frac{1}{4}\lambda_j T^c \bigg) = O(e^{-T^\varepsilon}).
\end{equation}

By Hoeffding's inequality, we have 
\begin{align}
    P\big(E_{2,j}^{(c)}\big) &\leq
    P\big(\Gamma(t_1',T) > \frac{1}{4} \lambda_j T^{b-a+c}\big) \nonumber\\
    &\leq \sum_{j\in J_{\lambda}} 
    P\left(\Big|\Lambda_j(t_1',T) - \lambda_j(T-l_1) \Big| > \frac{\lambda_j T^{b-a+c}}{4\kappa |J_{\lambda}|+1} \right) \nonumber\\
    &\leq 2|J_{\lambda}|\exp\bigg(-
    \frac{2\lambda_j^2 \cdot T^{2(b-a+c)}}{(4\kappa |J_{\lambda}|+1)^2 \cdot (T-l_1)}\bigg) \nonumber \\
    &= O(e^{-T^\varepsilon}). \label{eq:fd:case3:8}
\end{align}

In addition, we have
\begin{align}
    P\big(E_{2,j}^{(d)}\big) &\leq
    P\big(| \Delta_j(t_2',T) | > \frac{1}{4} \lambda_j T^{b-a+c}\big) \nonumber\\
    &\leq 2 \exp\bigg(-
    \frac{\lambda_j^2 \cdot T^{2(b-a+c)}}{8 \cdot (T-l_1)}\bigg)  \nonumber\\
    &= O(e^{-T^\varepsilon}).\label{eq:fd:case3:9}
\end{align}

Therefore, combining \eqref{eq:fd:case3:6}, \eqref{eq:fd:case3:7}, \eqref{eq:fd:case3:8} and \eqref{eq:fd:case3:9}, we have 
\begin{equation}\label{eq:fd:case3:11}
P(E_{2,j}^\complement \mid j\in(\calA\cup  \calR)^\complement) 
\leq P\big(E_{2,j}^{(a)}
\cup E_{2,j}^{(b)}
\cup E_{2,j}^{(c)} 
\cup E_{2,j}^{(d)} \mid j\in(\calA\cup  \calR)^\complement \big)
= O(e^{-T^\varepsilon}).
\end{equation}

Based on the results in \eqref{eq:fd:case3:10}
and \eqref{eq:fd:case3:11},
we then know in case 3) that
\begin{align}
P(E_j) &\geq P(E_j \mid j\in (\calA\cup \calR)^\complement) P(j \in (\calA\cup \calR)^\complement) \nonumber\\
&\geq \left(1-P(E_{1,j}^\complement \mid j\in (\calA\cup \calR)^\complement)-
P(E_{2,j}^\complement \mid j\in (\calA\cup  \calR)^\complement) \right)P(j \in (\calA\cup  \calR)^\complement)  \nonumber\\
&= 1- O(e^{-T^\varepsilon}). \label{eq:fd:case3}
\end{align}

Now combining the results in all three cases \eqref{eq:fd:case1}, \eqref{eq:fd:case2} and \eqref{eq:fd:case3}, we obtain $P(E_j) = 1-O(e^{-T^\varepsilon})$ for all $j\in[n]$,
given that $0 < \varepsilon < 2\gamma$, $\varepsilon\leq \beta$ and $\varepsilon \geq \frac{\log (\log T)}{\log T}$.
This completes the proof. \halmos

\end{proof}

\newpage

\section{Variants of Algorithms}\label{ap:resolve}
\subsection{Primal-dual restart algorithm}

Algorithm~\ref{alg:fr} shows the detailed procedure of our restart algorithm that does not require solving any LPs. The algorithm includes a total of $S+1$ \emph{epochs}. 
Each epoch $u$, where $u=0,1,\ldots, S$, starts from period $T-\tau_{u}+1$ and ends at period  $T-\tau_{u+1}$, where $\tau_u$ is the time length between the starting period of epoch $u$ and period $T$. At the beginning of each epoch, Algorithm~\ref{alg:fr} restarts Algorithm~\ref{alg:fd}.
The restart schedule is shown in Figure~\ref{fig:resolve}.


\begin{figure}[h]
    \centering
    \includegraphics[width=0.7\textwidth]{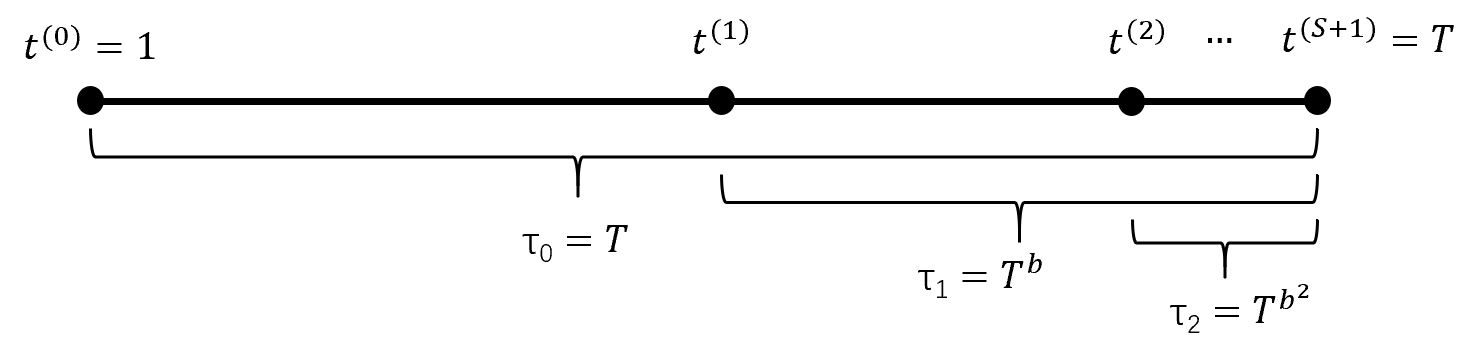}
    \caption{The restart schedule of Algorithm~\ref{alg:fr}}
    \label{fig:resolve}
\end{figure}

\begin{algorithm}[h]
\caption{Primal-Dual Restart ($\FR$)} \label{alg:fr}
\footnotesize
\begin{algorithmic}[1]
\STATE \textbf{Input:} time length $\tau_u$, parameters $\alpha_u$, $\beta_u$, $\gamma_u$ for $u=0,1,\ldots, S$.
\STATE \textbf{Initialize:}
set initial capacity $C(1)=C$;
\FOR{$u = 0,1,\ldots,S$}
\IF{$u < S$}
\STATE $C^{(u)} \gets C(T- \tau_u+1)$, $t_1^{(u)} \gets T- \tau_u + 1$ and  $t_2^{(u)} \gets T-\tau_{u+1}$
\ELSE{}
\STATE $C^{(u)}  \gets  C(T- \tau_u+1)$, $t_1^{(u)} \gets T- \tau_u+1$ and  $t_2^{(u)} \gets T$
\ENDIF
\STATE Run Algorithm~\ref{alg:fd} subroutine which begins with $t_1^{(u)}$ and ends with $t_2^{(u)}$. Let the initial capacity be  $C^{(u)}$, let technical parameters be $\alpha_u, \beta_u, \gamma_u$.
\ENDFOR
\end{algorithmic}
\end{algorithm}

\begin{algorithm}
\caption{LP-based Thresholding ($\LPT$)}
\footnotesize
\begin{algorithmic}[1]\label{alg:lpt}
\STATE \textbf{Input:} start time $t_1$, end time $t_2$, initial capacity at start time $B$; \\
\hspace{1cm} parameter $\beta$ where $\beta \in (0,\frac{1}{2})$, parameter $d$ where $d<(-\beta, \beta - \frac{1}{2}) $
\STATE \textbf{Initialize:} set  $b=\frac{1}{2}+\beta$, set $L=T-t_1+1$, $l_1 =\lceil L - L^b \rceil$, set $B(t_1) = B$, \\
\STATE calculate $x^* = \arg \max_x \left\{ \sum_{j\in[n]} r_j x_j  \mid  \sum_{j=1}^n A_j x_j \leq \frac{B(t_1)}{L}, 0 \leq x_j \leq \lambda_j, \, \forall\, j \in [n] \right\}$.
\FOR{$j \in [n]$}
\IF{$x^*_j < \lambda_j L^{d}$}  
\STATE set $p_j \gets 0$
\ELSIF{$x^*_j > \lambda_j (1 - L^{d})$}
\STATE set $p_j \gets 1$
\ELSE{}
\STATE set $p_j \gets x_j^{*}/\lambda_j$
\ENDIF
\ENDFOR
\STATE \textbf{Phase I.} Thresholding
\FOR{$t= t_1, t_1+1,\ldots, t_1 + l_1 -1$}
\STATE Observe customer of type $j(t)$
\IF{$A_j \leq B(t) \, \forall \, j \in [n]$}  
\STATE Accept the customer with probability $p_{j(t)}$;
\STATE If accepted, update $B(t+1) \gets B(t) - A_{j(t)}$
\ELSE{}
\STATE Reject the customer
\ENDIF
\ENDFOR
\STATE \textbf{Phase II.} Run Algorithm~\ref{alg:ff} subroutine with start time $t_1+l_1$, end time $t_2$, and initial capacity $B(t_1+l_1)$.
\end{algorithmic}
\end{algorithm}

\begin{algorithm}[h]
\caption{Hybrid Restart ($\MR$)} \label{alg:mr}
\footnotesize
\begin{algorithmic}[1]
\STATE \textbf{Input:}
time length $\tau_u$ for $u=0,1,\ldots,S$; \\
(Algorithm~\ref{alg:lpt} subroutine) parameters $\beta_u$ and $d_u$ for $u=0,1,\ldots,U$; \\
(Algorithm~\ref{alg:fd} subroutine) parameters $\alpha_u$, $\beta_u$ and $\gamma_u$ for $u=U+1,\ldots S$.
\STATE \textbf{Initialize:}
set initial capacity $C(1)=C$
\FOR{$u = 0,1,\ldots,S$}
\IF{$u < S$}
\STATE $C^{(u)} \gets C(T- \tau_u+1)$, $t_1^{(u)} \gets T- \tau_u + 1$ and  $t_2^{(u)} \gets T-\tau_{u+1}$
\ELSE{}
\STATE $C^{(u)}  \gets  C(T- \tau_u+1)$, $t_1^{(u)} \gets T- \tau_u + 1$ and  $t_2^{(u)} \gets T$
\ENDIF
\IF{$u < U$}
\STATE Run Algorithm~\ref{alg:lpt} with start time $t_1^{(u)}$, end time $t_2^{(u)}$, initial capacity $C^{(u)}$ at start time, \\
and parameters $\beta_u$, $d_u$.
\ELSE{}
\STATE Run Algorithm~\ref{alg:fd} which begins with $t_1^{(u)}$ and ends with $t_2^{(u)}$. Let the initial capacity be  $C^{(u)}$, let technical parameters be $\alpha_u, \beta_u, \gamma_u$.
\ENDIF
\ENDFOR
\end{algorithmic}
\end{algorithm}

\subsection{Hybrid restart algorithm}

In the following, we first introduce the LP-based thresholding algorithm in Algorithm~\ref{alg:lpt}, which is called as subroutines in 
Algorithm~\ref{alg:mr}.

Algorithm~\ref{alg:lpt} takes as inputs a start time $t_1$, an end time $t_2$, initial capacity $B$ at the start time, and technical parameters $\beta$ and $d$ that satisfy $\beta \in (0,\frac{1}{2})$ and $d<(-\beta, \beta - \frac{1}{2})$.

At the beginning of the time horizon, Algorithm~\ref{alg:lpt} solves a DLP, and uses its optimal solution to calculates the thresholding conditions. The entire horizon is then divided in two ``phases": in phase I, Algorithm~\ref{alg:lpt} makes probabilistic allocation decisions;
in phase II, Algorithm~\ref{alg:lpt} runs the primal-dual algorithm in Algorithm~\ref{alg:ff}.


Now we introduce Algorithm~\ref{alg:mr}, which is a hybrid of Algorithm~\ref{alg:lpt} and Algorithm~\ref{alg:fd} under a similar restart schedule shown in Figure~\ref{fig:resolve}.

Algorithm~\ref{alg:mr} includes a total of $S+1$ epochs. Each epoch $u$, where $u=0,1,\ldots, S$, starts from period $T-\tau_{u}+1$ and ends at period  $T-\tau_{u+1}$, where $\tau_u$ is the time length between the starting period of epoch $u$ and period $T$.
At the beginning of the first $U$ epochs, Algorithm~\ref{alg:mr} runs (restarts) Algorithm~\ref{alg:lpt} as a a subroutine, and 
at the beginning of the remaining $S-U+1$ epochs, Algorithm~\ref{alg:mr} runs (restarts) Algorithm~\ref{alg:fd} as a subroutine.

More precisely, we design the restart schedule as follows: in the first $U$ epochs, Algorithm~\ref{alg:mr} only runs phase I of Algorithm~\ref{alg:lpt}; in the next $S-U$ epochs, Algorithm~\ref{alg:mr} only runs phase I and phase II of Algorithm~\ref{alg:fd}; and in the last epoch, Algorithm~\ref{alg:mr} runs all three phases of Algorithm~\ref{alg:fd}.


\end{APPENDICES}

\end{document}